\documentclass[10pt]{amsart}

\usepackage{fancyhdr}
\usepackage{amsmath,amsthm, chngcntr}
\usepackage{amsfonts,amstext,amssymb, mathtools, comment, undertilde, mathrsfs, mathabx, graphicx, subcaption, array,longtable}
\RequirePackage{natbib}
\RequirePackage[colorlinks,citecolor=blue,urlcolor=blue]{hyperref}

\newcommand{\bfA}{\mathcal{A}}
\newcommand{\bs}{\boldsymbol}

\newcommand{\e}{{\mathbb E}}
\newcommand{\p}{\mathbb P}
\newcommand{\D}{\mathrm d}

\newcommand{\levy}{L\'{e}vy }

\newcommand{\R}{\mathbb R}
\renewcommand{\Re}{{\rm Re}}

\newcommand{\diag}{\mathrm {diag}}
\newcommand{\ind}[1]{\mbox{\rm\large  1}_{\{#1\}}}

\newcommand{\oo}[1]{\overline{#1}}

\renewcommand{\a}{\alpha}
\renewcommand{\b}{\beta}

\newcommand{\G}{\Gamma}

\newcommand{\s}{\sigma}

\newcommand{\bv}{{\bs v}}

\newcommand{\bh}{{\bs h}}

\newcommand{\bb}{{\bs b}}

\newcommand{\1}{{\bs 1}}

\newcommand{\bpi}{{\bs \pi}}

\newcommand{\matI}{\mathbb{I}}

\newcommand{\ee}{{\mathrm e}}
\def\te#1{\mathrm{e}^{#1}}

\newcommand{\ct}{{\bf The non-lattice case}}
\newcommand{\dt}{{\bf The lattice case}}

\newtheorem{theorem}{Theorem}

\newtheorem{cor}{Corollary}

\newtheorem{lemma}{Lemma}

\newtheorem{prop}{Proposition}

\newtheorem{remark}{Remark}

\title{Ruin Probabilities for Risk Process in a Regime Switching Environment}

\author[Z. Palmowski]{Zbigniew Palmowski}
\address{ Faculty of Pure and Applied Mathematics,
Wroc\l aw University of Science and Technology,
Wyb. Wyspia\'nskiego 27, 50-370 Wroc\l aw, Poland}
\email{zbigniew.palmowski@pwr.edu.pl}

\thanks{This work is partially supported by Polish National Science Centre Grant No. 2018/29/B/ST1/00756, 2019-2022}

\date{\today}

\begin{document}

\begin{abstract}
In this paper we give few expressions and asymptotics of ruin probabilities
for a Markov modulated risk process for various regimes
of a time horizon, initial reserves and a claim size distribution.
We also consider few versions of the ruin time.

\vspace{3mm}


\end{abstract}

\maketitle

\pagestyle{myheadings} \markboth{\sc Z.\ Palmowski
} {\sc Ruin Probabilities and Regime Switching}




\section{Introduction}\label{sec:iar}
This paper concerns  the ruin probabilities
in the context of regime switching environment. The risk theory is
substantial for insurance mathematics and has been subject
of many interest since Lundberg. There is a good deal of work related
with this theory and there are many great books describing it; see e.g.
\citet{asm_ruin, rolskibook, AndreasGerberShiu}. The understanding of so-called Markov modulated risk process is
also deep; see Chap. VII of \citet{asm_ruin}.
The need of Markov modulation is not questionable nowadays. It can model
the state of the roads in the case of car insurance, economic environments
or insurance claims with a higher degree of complexity insofar as claim frequencies and/or
severities are concerned; see e.g. \citet{Dawid} and references therein.

Still, we believe that there is a need of the short survey related with these results
and with few additional new facts that have not been proved yet.
In particular, we give counterpart of Arfwedson Theorem for the Markov modulated
risk process. We also give a subexponential asymptotics of the Parisian probability of ruin.
This is our first goal.
Second main goal is the ordering of few key results with the unified use of so-called scale matrices introduced
by \citet{ivanovs_palmowski}. This part of the paper can be then treated as a Markov modulated version of
some facts presented in \citet{AndreasGerberShiu}. Due to Markov modulation, some additional care is required here though.

The paper is organized as follows. In the next section we formally introduce the risk process that we work with.
In Section \ref{sec:prel} we give few main facts concerning the scale matrices.
In Section \ref{Ger:sec} we present the key compensation formula and
related representations of Gerber-Shiu function. Section \ref{sec:diff}
gives the ordinary differential equation that Gerber-Shiu function solves.
In Sections \ref{sec:Cram} and \ref{sec:subexpnentialasruinprob} we present the Cram\'er
and subexponential asymptotics, respectively. In Sections \ref{sec:seg} and \ref{sec:Hog}
we  focus on Segerdahl's and H\"{o}glund's approximations of the finite-time ruin probability.
The last section is dedicated to Parisian ruin probability.

\section{Gerber-Shiu function}
The goal of this paper is to analyse various ruin probabilities
of a Markov modulated risk process.
To describe this process properly we start from introducing a random environment that is
given by a continuous time Markov chain $J_t$ living on the state space $E=\{1, 2, \ldots N\}$.
Let $T_k$ denote successive jump epochs of the Markov chain $J_t$.
Then the risk process under consideration is given by
\begin{equation}\label{riskMM3}
X_t=x+\int_0^tp_{J_u}\D u -\sum_{k=1}^{N_t} C_k^{(J_t)} -\sum_{T_k\leq t} C_k^{(J_{T_k-}, J_{T_k})}.
\end{equation}
Above, $x$ describes an initial capital, $N_t$ is a Markov modulated Poisson process with an arrival intensity $\lambda_i$ at time $t$
when $J_t=i$ determining the arrival of i.i.d. claims $\{C_k^{(i)}\}$ which are conditionally independent of $N_t$ and having distribution
$F^{(i)}_C$
depending on the state $i\in E$ of the environmental Markov chain $J$ at time $t$.
Apart of it we have possible claims $C_k^{(ij)}$, $i,j\in E$, appearing when the environmental Markov chain changes its state.
The vector $(p_1, p_2, \ldots, p_N)$ is a vector of premium intensities.
In the following we assume that the processes $X$ and $J$ are defined on a common
filtered probability space $(\Omega, \mathcal{F}, \{\mathcal{F}_t\}_{t\ge 0}, \p)$
and we use $\p_{x,i}$ to denote the law of $(X,J)$ given
$\{X_0=x, J_0=i\}$ and by $\p_i$ and $\p_x$ we denote the law of $(X,J)$ given
$\{X_0=0, J_0=i\}$ and $\{X_0=x\}$, respectively.
The appropriate expectations we will denote by  $\e_{x,i}$, $\e_i$ and $\e_x$.
We also write $\e[Z;J_t]$, where $Z$ is some random
variable, to denote the $N\times N$ matrix with entries
$\e_i[Z;J_t=j]=\e_i[Z\ind{J_t=j}].$ Finally, we denote $\p(\cdot)=\p_0(\cdot)$ and $\e[\cdot]=\e_0[\cdot]$.

We will assume throughout of this work the following net profit condition
\begin{equation}\label{netprofit}
\e_{i}X_1>0,\qquad \text{for all $i\in E,$}
\end{equation}
under which risk process $X$ tends to infinity a.s.

Observe that $X_t$ is a spectrally negative Markov Additive Process (MAP) with the matrix exponent
\begin{equation}\label{Fforriskprocess}
F(\a)=
\diag \left(p_i\a +\lambda_i (\e e^{-C^{(i)}\a}-1)\right)_{\{i\in E\}} +(q_{ij}\e e^{-C^{(ij)}\a})_{\{i,j\in E\}},
\end{equation}
where $\alpha \ge 0$, $C^{(i)}$ and $C^{(ij)}$ are generic random variables representing
claims $C^{(i)}_k$ and $C^{(ij)}_k$, respectively, and $Q=(q_{ij})_{\{i,j\in E\}}$ is an intensity matrix of $J$.
In other words,
\[\e[e^{\a X_t};J_t]=e^{F(\a)t}.\]

The main object of the study is expected discounted penalty function (EDPF)
called a Gerber-Shiu function as well and defined by
\begin{equation}\label{GerberShiuFunction}
\phi_{w,ij}^{(q)}(x)=\e_{x,i} [e^{-q\tau_0^-}w(X_{\tau_0^--}, |X_{\tau_0^-}|); J_{\tau_0^-}=j, \tau_0^-<\infty]
\end{equation}
for the ruin time
\begin{equation}\label{ruintime}
\tau_0^-=\inf\{t\ge 0: X_t<0\}.\end{equation}
Above, $q\ge 0$ is a discounting factor and $w$ is a bivariate non-negative penalty function. This function describes the penalty that is paid at the moment of ruin.
It depends on the position prior ruin  $X_{\tau_0^--}$ and the deficit $|X_{\tau_0^-}|$ at the ruin moment.
If $w\equiv 1$ then Gerber-Shiu function is a discounted ruin probability:
\begin{equation}\label{ruinprobability}
\phi_{1,ij}^{(q)}=\e_{x,i} [e^{-q\tau_0^-}; J_{\tau_0^-}=j, \tau_0^-<\infty].\end{equation}
The most common case is when $q=0$ and then
\begin{equation}\label{ruinprobdef}
\phi_{ij}(x)=\phi_{1,ij}^{(0)}(x)=\p_{x,i}(\tau_0^-<\infty, J_{\tau_0^-}=j)\end{equation}
 is the ruin probability.
We will write $\phi_{w,i}^{(q)}(x)=\sum_{j\in E}\phi_{w,ij}^{(q)}(x)$, $\phi_{1,i}^{(q)}(x)=\sum_{j\in E}\phi_{1,ij}^{(q)}(x)$
and
\begin{equation}\label{ruinprobdef2}
\phi_{i}(x)=\sum_{j\in E}\phi_{ij}(x)=\p_{x,i}(\tau_0^-<\infty).
\end{equation}
In most of the cases $C^{(ij)}\equiv 0$ for any $i,j\in E$, that is, when
Markov chain changes its state there are no additional claims appearing in reserve process.
In this case
\begin{equation}\label{riskMM2}
X_t=x+\int_0^tp_{J_u}\D u -\sum_{k=1}^{N_t} C_k^{(J_t)},
\end{equation}

There is a huge amount of literature devoted to the Gerber-Shiu
functions and ruin theory. For more detailed discussions, we can refer
to \citet{asm_ruin, rolskibook, AndreasGerberShiu}.

\section{Prelimiaries}\label{sec:prel}
We follow here mainly~\citet{ivanovs_palmowski}.
Define an $N\times N$ matrix-valued function
$W^{(q)}(x)$ which is continuous for $x\ge 0$ and is identified by the
following Laplace transform:
\begin{equation}\label{Wq}
\int_0^\infty e^{-\a x}W^{(q)}(x)\D x=(F(\a)-q\matI)^{-1},\qquad \a>\max\{\Re(\lambda):\det (F(\lambda)-\lambda q\matI)=0\},\end{equation}
where $\matI$ is the identity matrix.
The matrix $W^{(q)}(x)$ is invertible for $x>0$ and satisfies
\begin{align}
&\e_x[e^{-q\tau_a^+}; \tau_x^+<\tau_{0}^-,J_{\tau_a^+}]=W^{(q)}(x)W^{(q)}(a)^{-1}\label{eq:two-sidedq}
\end{align}
for $0\le x\le a$,
where
$$\tau_a^+=\inf\{t\ge 0: X_t>a\}.$$
In the case of $N=1$ this results corresponds to~\citet[Thm.~1]{kyprianoupalmowski}.

We also define
\[Z^{(q)}(\a,x)=e^{\a x}\left(\matI-\int_0^xe^{-\a y}W^{(q)}(y)\D y(F(\a)-q\matI)\right)\text{ for }\a,q,x\ge 0\]
and when $\a=0$ we denote
\begin{equation}\label{secsc}
Z^{(q)}(x)=\matI-\int_0^xW^{(q)}(y)\D y(F(0)-q\matI)=\matI-\int_0^xW^{(q)}(y)\D y(Q-q\matI).
\end{equation}
It is so-called second scale matrix.
We will also need a matrix $R$
being a \emph{left solution} of $$F(-R^{(q)})=q\matI.$$
In case when $N=1$ we have
$R=-\Phi(q)$
where $\Phi(q)$ is a right-inverse of the Laplace exponent $\psi(\alpha)=F(\alpha)$ of spectrally negative L\'evy process $X$
(hence $\psi(\Phi(q))=q$).
Then for all $x,\a\ge 0$ and $q>0$
\begin{equation}\label{2sideddown}\e_x[e^{-q\tau_0^-+\a X_{\tau_0^-}};\tau_0^-<\tau_a^+,J_{\tau_0^-}]=Z^{(q)}(\a,x)-W^{(q)}(x)W^{(q)}(a)^{-1}Z^{(q)}(\a,a).\end{equation}
Hence by taking $a \rightarrow+\infty$ and assuming that $R^{(q)} +\a\matI$ is
non-singular, it holds that
\begin{equation*}
\e_x[e^{-q\tau_0^-+\a X_{\tau_0^-}};J_{\tau_0^-}]=Z^{(q)}(\a,x)-W^{(q)}(x)(R^{(q)}+\a\matI)^{-1}(F(\a)-q\matI).\end{equation*}
Above result can be extended to $q=0$ by taking
the limits as $q\downarrow 0$. Moreover, it is noted that
$(R^{(q)}+\a)^{-1}(F(\a)-q\matI)$ reduces to
$(\psi(\a)-q)/(\a-\Phi^{(q)})$ in the \levy case. This leads to the
known identity with $\a=0$ for a \levy process,
see~\citet[Thm.~8.1]{kyprianou}.

We can consider process $X$ killed on exiting positive half-line
and the corresponding potential measure (a matrix of measures)
\begin{align}\label{eq:potential}
 &U^{(q)}(x,A)=\int_0^\infty e^{-qt}\p_x[X_t\in A; t>\tau_0^-, J_t]\D t,
\end{align}
where $A$ is a Borel set.
It turns out that the measure $U^{(q)}(x,A)$ has a matrix density $u^{(q)}(x,z)$ on $(0,+\infty)$ with respect to Lebesgue measure
and by \citet{ivanovs_potential} it is given by
\begin{equation}\label{thm:13}
u^{(q)}(x,z)=W^{(q)}(x)e^{R^{(q)}z}-W^{(q)}(x-z).\end{equation}

\section{Compensation formula and Gerber-Shiu function}\label{Ger:sec}
Our first goal is to give various key representations of Gerber-Shiu function.
We start from describing so-called compensation formula in the context of
regime switching ruin probabilities.

We define (non-homogeneous) marked Poisson measure
\[N_C([0,t)\times A)=\sum_{k=1}^{N_t}\delta_{C_k^{(J_{T^N_k})}}(A) + \sum_{T_k\leq t}\delta_{C_k^{(J_{T_k-}, J_{T_k})}}(A),\]
where $T^N_k$ are jumps epochs of the Poisson process $N_t$, $\delta_x$ is the Dirac measure at $x$ and $A$ is a Borel set.
From \citet[Def. II.1.20 and Prop. II.1.21]{JacodShiryaev} and \citet{Cinlar1} it follows that
\begin{equation}\label{nu}
\nu^{(J_{s-}, J_s)}(\D y)=\lambda_{J_s}F^{(J_s)}_C(\D y) + q_{J_{s-}, J_s} F_C^{(J_{s-}, J_s)}(\D y)
\end{equation}
is a compensator of $N_C$, that is, for any measurable and bounded $g$ we have
\[\e_{i} \int_0^t g(s,y, J_{s-}, J_s) N_C(\D s, \D y)=\e_{i} \int_0^t g(s,y, J_{s-}, J_s)
\nu^{(J_{s-}, J_s)}(\D y)\D s.\]
We denote by $T_n^C$ the jump epoch of the risk process \eqref{riskMM3}. This means
that $T_n^C$  is either $T^N_k$ (the jump epoch of the Poisson process)
or $T_n$ (the moment when Markov chain $J$ changes a state).
This gives the following compensation formula.
\begin{theorem}\label{compformula}
For any measurable and bounded function $f: [0,+\infty)^2\times\mathbb{R}\times E^2\rightarrow \mathbb{R}$ we have
\begin{eqnarray*}
\lefteqn{\e_{x,i}\left[\sum_{T_k^C\leq t} f(T_k^C, X_{T^C_k-}, X_{T^C_k}, J_{T^C_k-}, J_{T^C_k})\right]}\\
&&=\sum_{k,j\in E}\int_0^\infty\e_{x,i}\int_0^t f(s, X_{s}, X_{s}-y, k, j)
\D s\; \nu^{(k,j)}(\D y).\end{eqnarray*}
\end{theorem}
\begin{proof}
Note that
\begin{align*}
&\e_{x,i}\left[\sum_{T^C_k\leq t} f(T^C_k, X_{T^C_k-}, X_{T^C_k}, J_{T^C_k-}, J_{T^C_k})\right]\\
&=\e_{x,i}\left[\sum_{T^C_k\leq t} f(T^C_k, X_{T^C_k-}, X_{T^C_k-}
-C_k^{(J_{T_k^C})}\ind{J_{T_k^C-}=J_{T_k^C}}-C_k^{(J_{T_k^C-}, J_{T_k^C})}\ind{J_{T_k^C-}\neq J_{T_k^C}}, J_{T^C_k-}, J_{T^C_k})\right]\\
&=\e_{x,i} \int_0^t f(s, X_{s-}, X_{s-}-y, J_{s-}, J_s)N_C(\D s, \D y)\\&
=\sum_{k,j\in E}\int_0^\infty\e_{x,i}\int_0^t f(s, X_{s}, X_{s}-y, k, j)
\D s\; \nu^{(k,j)}(\D y)
\end{align*}
which completes the proof.
\end{proof}
\begin{cor}\label{compensationGB1}
We have
\[\phi_{w,ij}^{(q)}(x)=\sum_{k\in E} \int_0^\infty \int_0^y w(z,y-z)u^{(q)}_{ik}(x,z) \D z\; \nu^{(k,j)}(\D y),\]
where $i,j\in E$ and $u^{(q)}(x, z)$ is a $q$-potential density of the process $X$ starting at $x$ and killed on exiting from $[0,+\infty)$.
\end{cor}
\begin{proof}
Observe that ruin can happen only at the moments of claim arrivals. Hence
from Theorem \ref{compformula},
\begin{eqnarray*}
\lefteqn{\phi_{w,ij}^{(q)}(x)=\e_{x,i}\left[\sum_{k=1}^\infty e^{-q T^C_k}w(X_{T_k^C-},|X_{T_k^C}|)\ind{X_{T_k^C-}\ge 0, X_{T_k^C}<0, J_{T_k^C}=j}\right]}\\
&&\sum_{k\in E} \int_0^\infty \left(
\e_{x,i} \left[\int_0^\infty \int_0^y e^{-q s}w(z,y-z)\ind{X_{s-}\in \D z, \tau_0^->s}\ind{J_{s-}=k, J_s=j}\D s\right]
\right)
\nu^{(k,j)}(\D y)\\
&&=\sum_{k\in E} \int_0^\infty \left(
\left[\int_0^\infty \int_0^y e^{-q s}w(z,y-z)\p_{x,i}(X_{s-}\in \D z, \tau_0^->s, J_{s-}=k)
\D s\right]
\right)
\nu^{(k,j)}(\D y).
\end{eqnarray*}
This completes the proof.
\end{proof}

Writing $\mathbf{\phi}_{w}^{(q)}(x)=(\phi_{w,1}^{(q)}(x), \ldots, \phi_{w,N}^{(q)}(x))^T$ and using Fubbini theorem
we can conclude that Corollary \ref{compensationGB1} can be rewritten in the matrix notation as follows:
\begin{align}\label{compstarting}
\mathbf{\phi}_{w}^{(q)}(x)&= \int_0^\infty \int_0^y w(z,y-z)u^{(q)}(x,z)\D z\; \nu (\D y)\nonumber\\
&=\int_0^\infty \int_0^\infty w(z,y)u^{(q)}(x,z)\; \nu (z+\D y)\D z
\end{align}
where the measure $\nu$ is defined formally in \eqref{nu}.

\begin{remark}
\rm For above considerations the downward jumps are not crucial and
above corollary could be easily adopted for the general direction of jumps
of the process $X$.
\end{remark}

Similar result was derived by \citet{SalahMorales} and \citet[Thm. V.5.5]{AndreasGerberShiu} in the context of spectrally negative
L\'evy processes.
From Corollary \ref{compensationGB1} and \eqref{thm:13}
we have the first main result of this section.
\begin{theorem}[Compensation representation]\label{compensationrepresentation}
We have
\[\mathbf{\phi}_{w}^{(q)}(x)=\int_0^\infty \int_0^\infty w(z,y)u^{(q)}(x,z)\; \nu (z+\D y)\D z,\]
where $u^{(q)}(x,z)$ is given in \eqref{thm:13}.
\end{theorem}

We recall that the matrix $\mathbf{\phi}_1^{(q)}(x)=(\mathbf{\phi}_{1,ij} ^{(q)}(x))_{\{i,j\in E\}}$
describes the Laplace transform of the ruin time on the event when ruin happens and it is
defined formally in \eqref{ruinprobability}.
From \eqref{2sideddown} by taking $\alpha=0$ and $a\rightarrow +\infty$
the following its representation holds true (compare with \citet[Thm. IV.4.3]{AndreasGerberShiu}).
\begin{theorem}[Discounted ruin probability]\label{scaleruinproabilitydiscounted}
We have,
\begin{equation}\label{LTruintime}
\mathbf{\phi}_1^{(q)}(x)
=Z^{(q)}(x)- W^{(q)}(x)\mathbf{C}_{W(\infty)^{-1}Z(\infty)},
\end{equation}
for
\begin{equation*}
 \mathbf{C}_{W(\infty)^{-1}Z(\infty)}=
 \lim_{a \to \infty} W^{(q)}(a)^{-1}Z^{(q)}(a)<\infty.
\end{equation*}
\end{theorem}

Observe that $\mathbf{C}_{W(\infty)^{-1}Z(\infty)}$ is well-defined since
$\mathbf{\phi}_1^{(q)}(x)$ as a limit is well-defined.
Note that identity \eqref{LTruintime} gives the Laplace transform of the finite-time ruin time as well, that is,
\[\mathbf{\phi}_1^{(q)}(x)=\int_0^\infty e^{-q t}\D \p_{x} (\tau^-_0\leq t).
\]
In other words inverting \eqref{LTruintime} gives the finite-time ruin probability
\begin{equation}\label{finitetimeruinprob}
\mathbf{\phi}(x,t)=\p_{x} (\tau^-_0\le t).
\end{equation}

Multiplying by $\1$ from the right and taking $q\downarrow 0$ in \eqref{LTruintime} gives also the representation of the ruin probability
$\phi_i(x)$ defined in \eqref{ruinprobdef2}.
Indeed, recalling that $\phi(x) =(\phi_{ij}(x))_{\{i,j\in E\}}$ is defined in \eqref{ruinprobdef}, note that
\begin{align*}
&\phi_i(x)=(\phi(x)\1)_i=
\p_{x,i}(\tau_0^-<\infty)\\&
=(\matI \1-\int_0^xW(y)\D y F(0)\1 - W(x)\mathbf{C}_{W(\infty)^{-1}Z(\infty)}\1)_i=
1-(W(x)\mathbf{C}_{W(\infty)^{-1}Z(\infty)})\1)_i.\end{align*}
Above we have used the definition of the second scale matrix
$Z^{(q)}(x)$ given in \eqref{secsc} and fact that $F(0)\1=Q\1=0$.
We introduce now the survival probability
\[\overline{\phi}_i(x)=1-\phi_i(x).\]
Using fact that matrix $\mathbf{C}_{W(\infty)^{-1}Z(\infty)}$  is well-defined
we can conclude the following representation of survival probability.
\begin{theorem}[Survival probability]\label{scaleruinproability}
Under net profit condition \eqref{netprofit} we have
\begin{equation}\label{ruintimeprob}
\overline{\phi}_i(x)=(W(x)\mathbf{C}_{W(\infty)^{-1}Z(\infty)}\1)_i, \qquad i\in E.
\end{equation}
\end{theorem}

Let now $N=1$. Then $Q=0$.
Recall that
\[\int_0^\infty e^{-\a x}W(x)\D x=F(\a)^{-1}.\]
In our case by \eqref{Fforriskprocess},
\begin{align*}F(\a)^{-1}=&
\frac{1}{p_1\a } \frac{1}{1-\lambda_1\e C^{(1)}_1\int_0^\infty e^{-\a y}\overline{F}_C^{(1)}(y)/\e C^{(1)}_1\D y}
\\
&=
\frac{1}{p_1\a } \sum_{k=1}^\infty \rho_1^k \left(\int_0^\infty e^{-\a y} F_C^{I,(1)}(\D y)\right)^k
\end{align*}
where \[\rho_1=\frac{\lambda_1\e C^{(1)}_1}{p_1},\]
\[F^{I,(1)}_C(x)=\frac{1}{\e C^{(1)}_1}\int_0^x\overline{F}_C^{(1)}(y)\D y,\]
for
\[\overline{F}_C^{(1)}(x)=1-F_C^{(1)}(x).\]
Thus
\[W(x)=
\sum_{k=1}^\infty \rho_1^k (F^{I,(1)}_C)^{*k}(x)
,\]
where $(F^{I,(1)}_C)^{*k}$ denotes the $k$th convolution of distribution
$F^{I,(1)}_C$.
To sum up,
\[\overline{\phi}_1(x)=A_1\sum_{k=1}^\infty \rho_1^k (F^{I,(1)}_C)^{*k}(x)\]
for some constant $A_1$.
From the fact that $\lim_{x\to+\infty}\overline{\phi}_1(x)=1$
we can get identify constant $A_1$ and derive the seminal Pollaczek-Khintchine formula for the survival probability.
\begin{theorem}[Pollaczek-Khintchine formula] If $N=1$ then
\[\overline{\phi}_1(x)=(1-\rho_1)\sum_{k=1}^\infty \rho_1^k (F^{I,(1)}_C)^{*k}(x).\]
\end{theorem}

We denote by
\[\xi_k=X_{T^C_{k-1}}-X_{T^C_{k}}\]
the negative of the increments of the risk process between consecutive jumps.
Note that, conditionally on $J_{T_l^C}$, random variables
$\xi_k$ are independent on each other. Moreover, conditionally on
$\{J_{T_k^C}=j, J_{T_{k-1}^C}=i\}$, $-\xi_k$ has the same law
as
\begin{equation}\label{jumprepresentationdistribution}
p_i{\rm Exp}(\lambda_i-q_{ii})- \frac{\lambda_i}{\lambda_i-q_{ii}}C^{(i)}_k-
\frac{q_{ij}}{\lambda_i-q_{ii}}C^{(ij)}_k\end{equation}
because waiting time for next jump has exponential law ${\rm Exp}(\lambda_i-q_{ii})$ with parameter
$q_{ij}+\lambda_i$. Then we have to take into account which jump is first: the one coming from Markov modulated Poisson process
$N$ or the one that Markov chain change of the state brings.
In the first scenario Markov chain $J_t$ remains in the state $i$.
Note that the law of $\xi_k$ depends on the states $J^C_{k-1}$ and $J_k^C$ of the discrete-time Markov chain
\[J_0^C=J_0,\qquad J^C_k=J_{T_k^C}, \qquad k=1,2,\ldots\]
with the transition matrix
\begin{equation}\label{transition}
\mathscr{P}=(p_{ij})\quad \text{with}\quad p_{ii}=\frac{\lambda_i}{\lambda_i-q_{ii}}\quad\text{and}
\quad p_{ij}=\frac{q_{ij}}{\lambda_i-q_{ii}}\quad\text{for $i\neq j$.}
\end{equation}
We define the following discrete-time Markov modulated random random walk:
\begin{equation}\label{MMRW}
S_0=0, \qquad S_n=\sum_{k=1}^n \xi_k.
\end{equation}

We recall the key observation that ruin of the risk process
\eqref{riskMM2} only at the claim arrivals. By shifting a generic trajectory
of $X_t$ by $x$ units downward and then reflecting it across the horizontal
axis, we can observe that the risk process $X$ gets below zero level if and only if
$S_k$ will ever cross level $x$. Moreover, the net profit condition \eqref{netprofit}
is equivalent to the requirements that $S_n\rightarrow+\infty$ a.s.
This gives us the last representation.
\begin{theorem}[Maximum random walk]\label{MMRW}
We have
\begin{equation}\label{maxMMRW}
\phi_i(x)=\p_i\left(\max_{k\ge 0} S_k>x\right)<1.\end{equation}
\end{theorem}
\begin{remark}\label{Markovmodulatedrandomwalkremarkassump}\rm
Note that
by enriching the state space to the pairs $(i,j)\in E^2$ and by taking the Markov chain
$\tilde{J}_t=(J_{t-},J_t)$, we can assume without loss of
generality that the distribution of generic increment $\xi$, depends only on the state
of the Markov chain just prior jump. In this case $X$ becomes so-called non-anticipative MAP.
In other words, without loss of generality we can assume
the distribution of $\xi$ (being the negative of r.v. given in \eqref{jumprepresentationdistribution})
can be indexed by
$F_\xi^{(i)}(x)$.
To simplify this analysis we will assume from now on that this condition holds whenever we work with random walk $S_n$.
In this case, denoting
\[a_i=\int_{\mathbb{R}} x F_\xi^{(i)}(\D x)\]
which we assume to be finite,
the net profit condition \eqref{netprofit} is equivalent to
\begin{equation}\label{driftMMRW}
\overline{a}=-\sum_{i=1}^N \pi_ia_i>0.
\end{equation}
\end{remark}

\begin{remark}\rm
The representation \eqref{maxMMRW} of the ruin probability via maximum random walk remains
true even if the generic time between arrivals of the claims
has a general distribution possibly depending on the discrete time Markov chain $J^C$.
\end{remark}

\section{Ordinary differential equation}\label{sec:diff}
In this section we derive ordinary differential equation for Gerber-Shiu
functions.
We start from proving some crucial martingale
property.
\begin{theorem}\label{scalemartingales}
For each $i\in E$, the processes:
\[e^{-qt\wedge \tau_0^-\wedge \tau_a^+}W^{(q)}_{iJ_{t\wedge \tau_0^-\wedge \tau_a^+}}(X_{t\wedge \tau_0^-\wedge \tau_a^+}),\qquad e^{-qt\wedge \tau_0^-\wedge \tau_a^+}Z^{(q)}_{iJ_{t\wedge \tau_0^-\wedge \tau_a^+}}(X_{t\wedge \tau_0^-\wedge \tau_a^+})\]
are uniformly integrable  martingales with respect of the natural filtration
$\mathcal{F}_t$ of MAP $(X,J)$.
\end{theorem}
\begin{proof}
From \eqref{eq:two-sidedq} and the strong Markov property it follows that
\begin{eqnarray*}
\lefteqn{\e_{x,i}\left[e^{-q\tau_a^+}\ind{\tau_a^+<\tau_0^-, J_{\tau_a^+}=j}|\mathcal{F}_{t\wedge \tau_0^-\wedge \tau_a^+}\right]}
\\&&= e^{-qt\wedge \tau_0^-\wedge \tau_a^+}W^{(q)}_{iJ_{t\wedge \tau_0^-\wedge \tau_a^+}}(X_{t\wedge \tau_0^-\wedge \tau_a^+})[W^{(q)}(a)^{-1}]_{J_{t\wedge \tau_0^-\wedge \tau_a^+}j},\end{eqnarray*}
where we used the fact that $W^{(q)}(X_{\tau_0^-})=0$ and $W^{(q)}(X_{\tau_a^+})W^{(q)}(a)^{-1}=\matI$.
This completes of the proof of the first martingale property since $[W^{(q)}(a)^{-1}]_{J_{t\wedge \tau_0^-\wedge \tau_a^+}j}$
is $\mathcal{F}_{t\wedge \tau_0^-\wedge \tau_a^+}$-measurable.

Similarly, from \eqref{2sideddown}
we have
\begin{eqnarray*}
\lefteqn{\e_{x,i}\left[e^{-q\tau_0^-}\ind{\tau_0^-<\infty, J_{\tau_0^-}=j}|\mathcal{F}_{t\wedge \tau_0^-\wedge \tau_a^+}\right]}
\\&&= e^{-qt\wedge \tau_0^-\wedge \tau_a^+}\left(Z^{(q)}_{iJ_{t\wedge \tau_0^-\wedge \tau_a^+}}(X_{t\wedge \tau_0^-\wedge \tau_a^+})
-W^{(q)}_{iJ_{t\wedge \tau_0^-\wedge \tau_a^+}}(X_{t\wedge \tau_0^-\wedge \tau_a^+})A_{J_{t\wedge \tau_0^-\wedge \tau_a^+}j}\right)
\end{eqnarray*}
for some matrix $A$. Using similar arguments as above we derive the martingale property of the second process.
\end{proof}

From the compensation representation of the Gerber-Shiu function given in Theorem
\ref{compensationrepresentation} we can conclude that
\begin{equation}\label{martgerbere}
e^{-qt\wedge \tau_0^-\wedge \tau_a^+}\phi^{(q)}_{w,i,J_{t\wedge \tau_0^-\wedge \tau_a^+}}(X_{t\wedge \tau_0^-\wedge \tau_a^+})\end{equation}
is a martingale as well.

Moreover, from Theorem 5 of \citet{ivanovs_palmowski} (see remark just after this theorem as well)
and by Lemma 2.4 of \citet{KKR}
we know that
if $F_C^{(i)}$ and $F_C^{(ij)}$ ($i,j\in E$) are absolutely continuous, then
$W^{(q)}\in \mathcal{C}^{1}(0,\infty)$.
Which means from Theorem \ref{compensationrepresentation}
that $\phi_w\in \mathcal{C}^{1}(0,\infty)$.
Thus $\phi_w$ is sufficiently smooth
to apply infinitesimal generator.
\begin{theorem}\label{ODEGerberShiu}
Assume that for each $i,j\in E$ the distribution functions
$F_C^{(i)}$ and $F_C^{(ij)}$ have continuous densities $f_C^{(i)}$ and $f_C^{(ij)}$.
Then
the Gerber-Shiu function $\phi_w^{(q)}$ is in $\mathcal{C}^{1}(0,\infty)$
and for $x\ge 0$ it satisfies the following differential equation
\begin{eqnarray}
\lefteqn{p_i\left(\phi_{w,i}^{(q)}\right)^\prime(x)
+\int_0^\infty(\phi_{w,i}(x-y)-\phi_{w,i}(x))f_C^{(i)}(y) \D y}\nonumber\\&&
+\sum_{k\in E \atop k\neq i} q_{ik}
\int_0^\infty(\phi_{w,i}^{(q)}(x-y)-\phi^{(q)}_{w,i}(x))f_C^{(ik)}(y) \D y - q \phi^{(q)}_{w,i}(x)=0\label{diff1}
\end{eqnarray}
with the boundary conditions
\begin{align*}
\phi_w^{(q)}(x)&=w(x)\quad\text{for $x<0$,}\\
\lim_{x\to+\infty}\phi_w^{(q)}(x)&=0.\end{align*}
\end{theorem}
\begin{proof}
Since $a>0$ in \eqref{martgerbere} is general then \eqref{diff1} follows straightforward
from \eqref{martgerbere} and Dynkin formula since it is equivalent to the requirement that
\[\bfA \phi_{w,i}^{(q)}(x)-q\phi_{w,i}^{(q)}(x)=0,\qquad i\in E,\]
where $\bfA$ is the infinitesimal generator of $X$ with the domain included in $\mathcal{C}^{1}(0,\infty)$.
It suffices to prove now only the boundary conditions.
The first one follows straightforward from definition
of the Gerber-Shiu function.
The second one is a consequence of net-profit condition \eqref{netprofit}.
\end{proof}

The equation given in Theorem \ref{ODEGerberShiu}
is well-known for classical risk process; see e.g. \citet{AndreasGerberShiu} and references therein.
In the context of Markov modulated risk process it appears
e.g. in \citet{asmu, Jacobsen, NgYang, BadescuDavid, CheungDavid}.

\section{Cram\'er asymptotics}\label{sec:Cram}
We will follow \citet{asm_ruin} and \citet{asmu}.

We recall that for $\a\ge 0$ the matrix $F(\a)$ has a real simple eigenvalue
$k(\a)$, which is larger than the real part of any other
eigenvalue. The corresponding left-eigenvector $\bv(\a)$ and
right-eigenvector $\bh(\a)$ can be chosen so that $v_i(\a)>0$ and
$h_i(\a)>0$ for all $i$. The normalization requirement
\begin{align*}&\bpi\bh(\a)=1,&\bv(\a)\bh(\a)=1\end{align*}
results in the unique choice of $\bv(\a)$ and $\bh(\a)$, where
$\bpi=(\pi_1, \pi_2, \ldots, \pi_N)$ is a stationary distribution of $J$.
Observe that $k(0)=0$, $\bh(0)=\1$ and $\bv(0)=\bpi$.

We assume in this section that there exists solution $\gamma >0$ of so-called Cram\'er-Lundberg equation
\begin{equation}\label{cramerlundbergequation}
k(-\gamma)=0.
\end{equation}
Note that a necessary condition for this existence
is that $\int_0^\infty e^{\gamma x} \p(C^{(i)}\in \D x)<+\infty$
and that $\int_0^\infty e^{\gamma x} \p(C^{(ij)}\in \D x)<+\infty$, hence both
$C^{(i)}$ and $C^{(ij)}$ must be light-tailed.
This solution $\gamma$ is called an adjustment coefficient.

We consider now the exponential change of measure
\begin{equation}\label{expomeasureruin}
\frac{\D \tilde{\p}_{|\mathcal{F}_t}}{\D\p_{|\mathcal{F}_t}}=e^{-\gamma (X_t-x)-k(-\gamma)t}
\frac{h_{J_t}(\gamma)}{h_{J_0}(\gamma)}
=e^{-\gamma (X_t-x)}
\frac{h_{J_t}(\gamma)}{h_{J_0}(\gamma)}.\end{equation}
From \citet{jabernoulli} it follows that our risk process $(X,J)$
under $\tilde{\p}$ is again MAP with
the matrix exponent
\[\tilde{F}(\a)=\Delta_{\bh(-\gamma)}^{-1}F(\a -\gamma)\Delta_{\bh(-\gamma)}. \]
Thus the largest eigenvalue under $\tilde{\p}$ equals
$\tilde{k}(\a)=k(\a-\gamma)$ and hence
$\tilde{k}^\prime(0)=k^\prime(-\gamma)$.
Moreover, from the martingale property of the density process used in the above change of measure  we can conclude that
\[\e_i\left[e^{\a X_t}\frac{1}{h_{J_t}(\a)}\right]=
e^{k(\a)t}\frac{1}{h_{i}(\a)}.\]
By twice differentiation, we derive
\[ {\rm Var}_{{\bs \pi}} X_1=k^{''}(0)\ge 0\]
(see e.g. \citet[Cor. 2.6]{APQ}), where
${\rm Var}_{{\bs \pi}}$ is a variance taken with respect of $\e_{{\bs \pi}}:
=\sum_{i=1}^N \pi_i \e_i$. Thus
the function $\beta \rightarrow k(\beta)$ is convex.
We have
\begin{align}\label{eq:kappa} k'(0)
=\e_\bpi X_{1}.\end{align}

By net profit condition \eqref{netprofit} we have $k^\prime(0)>0$ and therefore
this means that $\tilde{k}^\prime(0)=k^\prime(-\gamma)<0$.
From \eqref{eq:kappa} we know that $\tilde{k}^\prime(0)$
equals the asymptotic drift of $X$ under $\tilde{\p}$ which is negative,
that is, $\lim_{t\to+\infty} X_t=-\infty$ $\tilde{\p}$-a.s.
In other words, under new measure the ruin is certain:
\begin{equation}\label{changedrift}
\tilde{\p}(\tau_0^-<\infty)=1.
\end{equation}
This allows us to prove the main theorem of this section.
\begin{theorem}\label{Crameasymptotics}
\[\lim_{x\to+\infty}\frac{\phi_i(x)}{e^{-\gamma x}}=C_i\]
for some finite constant $C_i>0$ and $i\in E$.
\end{theorem}
 \begin{proof}
Denoting by $\tilde{\e}$ the expectation with respect of $\tilde{P}$
note that by Optional Stopping Theorem we have
\[\phi_i(x)=e^{-\gamma x}\tilde{\e}_{x,i}\left[e^{\gamma X_{\tau_0^-}}\frac{h_{i}(\gamma)}{h_{J_{\tau_0^-}}(\gamma)};\tau_0^-<\infty\right]=
e^{-\gamma x}\tilde{\e}_{x,i}\left[e^{\gamma X_{\tau_0^-}}\frac{h_{i}(\gamma)}{h_{J_{\tau_0^-}}(\gamma)}\right]\]
where the last equality follows from \eqref{changedrift}.
Moreover, using dual process $\hat{X}=-X$,
we have
\[\frac{\phi_i(x)}{e^{-\gamma x}}=\tilde{\e}_{0,i}\left[e^{\gamma (\hat{X}_{\hat{\tau}_x^+}-x)}\frac{h_{i}(\gamma)}{h_{J_{\hat{\tau}_x^+}}(\gamma)}\right],\]
where $$\hat{\tau}_x^+=\inf\{t\ge 0: \hat{X}_t>x\}.$$
 From the Renewal Theorem 28 of
 \citet{dereich2015real} (see also \citet{Athreyarenewal},
\citet{lalley}, \citet{kesten2} and \citet{Alsmeyer})
 we can conclude that
 right hand side of above identity tends to constant.
\end{proof}

\begin{remark}\rm
In the context of general L\'evy processes above Cram\'er asymptotics was proved by \citet{BerDonCramer}.
\end{remark}

\section{Subexponential asymptotics}\label{sec:subexpnentialasruinprob}
We recall that by Theorem \ref{MMRW} the ruin probability $\phi_i(x)$ defined formally in \eqref{ruinprobdef}
equals the tail of the distribution of the
maximum \[M=\sup_{k\ge 0} S_k\]
of a Markov modulatated random walk
\begin{equation}\label{Sk}S_k=\sum_{l=1}^k \xi_k\end{equation}
with negative drift defined in \eqref{MMRW}
where the distribution $F_\xi^{(i)}(x)$ of $\xi$ is determined by the random variable
\eqref{jumprepresentationdistribution} and discrete time Markov chain with the transition matrix \eqref{transition}.
This section deals with the study of the asymptotic distribution of the tail of the distribution of $M$
when the increments $\xi_k$ have heavy-tailed
distributions, that is, when solution of Cram\'er-Lundberg equation \eqref{cramerlundbergequation} does not exist.
By a \emph{heavy-tailed}
distribution we mean a distribution (function) $G$ on $\R$
possessing no exponential moments: $\int_0^\infty e^{sy}G(\D y)=\infty$
for all $s > 0$.
We will use the {\em principle of a single big jump},
which says that the maximum of the random
walk is essentially due to a single very large jump.
More precisely, in this section we will
model the claim size by a subexponential distribution.
This family of distributions is used to model many catastrophic events like earthquakes, storms, terrorist attacks
etc. Additionally, insurance companies use e.g. the lognormal distribution (which is subexponential) to model car claims.

For any distribution function $G$ on $\R$, we set
$\oo{G}(x)={}1-G(x)$ and denote by
$G^{*n}$ the $n$-fold convolution of $G$ by itself.
A distribution $G$ on $\R_+$ belongs to the class~$\mathcal{S}$ of
\emph{subexponential} distributions if and only if, for all $n\ge2$,
we have
\[\lim_{x\to\infty}\oo{G^{*n}}(x)/\oo{G}(x)=n.\]
It is sufficient to verify this condition in the case~$n=2$ - see \citet{CHIST}.
This statement is easily shown to be equivalent to the
condition that, if $\xi_1,\dots,\xi_n$ are i.i.d.\ random variables with
common distribution $G$, then
\begin{displaymath}
  \p(\xi_1+\cdots+\xi_n > x) \sim \p(\max(\xi_1,\ldots,\xi_n) > x),
\end{displaymath}
a statement which already exemplifies the principle of a single big
jump. Here for any two functions $f$, $g$ on $\R$, by
$f(x)\sim{}g(x)$ as $x\to\infty$ we mean
$\lim_{x\to\infty}f(x)/g(x)=1$; we also say that $f$ and $g$ are
\emph{tail-equivalent}. The class~$\mathcal{S}$
includes all the heavy-tailed distributions commonly found in
applications, in particular regularly-varying, lognormal and Weibull
distributions.
%
%
%
For any distribution~$G$ on $\R$ with finite mean, we define the
integrated (or second) tail distribution (function)~$G^{\text{\tiny \rm I}}$ by
\begin{displaymath}
  \oo{G^{\text{\tiny \rm I}}}(x) = 1-G^{\text{\tiny \rm I}}(x)
  = \min\left(1,\int_x^\infty\oo{G}(z)\,\D z\right).
\end{displaymath}
Good surveys of the basic properties of heavy-tailed
distributions, in particular long-tailed and subexponential
distributions, may be found in \citet{FDZ}, \citet{EKM} and in \citet{asm_ruin}.

In this section
we assume that there exists some \emph{reference}
distribution~$F$ with finite mean and some constants $c_i\ge 0$ $(i\in E$) such that
\begin{flalign*}
  \text{(D1)} & \quad
  \oo{F_\xi^{(i)}}(x)  \le \oo{F}(x),
  \quad \text{ for all } x \in \R, \qquad i\in E, &\\
  \text{(D2)} & \quad
  \oo{F^{(i)\text{\tiny \rm I}}_\xi}(x)  \sim c_i \oo{F^{\text{\tiny \rm I}}}(x)
  \quad \text{ as } x \to \infty, \qquad i\in E,&\\
  \text{(D3)} & \quad F^{\text{\tiny \rm I}}\in\mathcal{S}.
\end{flalign*}

The condition (D1) is no less restrictive than the condition
\[
\limsup_{x\to\infty} \max_{i\in E} \frac{\oo{F^{(i)}_\xi}(x)}{\oo{F}(x)} < \infty,
  \]
in which case it is straightforward to redefine $F$, and then $c$,
so that (D1) and (D2) hold as above.
Further, the condition  ${(D3)}$ holds for example when integrated tail distribution of claim size
at some state $i\in E$ is subexponential, that is there exists $i\in E$ such that
\[F_C^{(i)I}\in \mathcal{S}.\]

Define \begin{equation}
  \label{eq:3}
  C_S = \sum_{i\in E} c_i\pi_i.
\end{equation}
In \citet{FKZ} the following result is proved.
\begin{theorem}
  \label{simple}
  Suppose that (D1)--(D3) hold. 
  Then
\begin{displaymath}
  \lim_{x\to\infty}\frac{\phi_i(x)}{\oo{F^{\text{\tiny \rm I}}}(x)}
 =\lim_{x\to\infty}\frac{\p_i(M>x)}{\oo{F^{\text{\tiny \rm I}}}(x)}
  = \frac{C_S}{\overline{a}},
\end{displaymath}
where $\overline{a}$ is given in \eqref{driftMMRW} and $i\in E$.
\end{theorem}

Most of the papers concern simple random walk hence the case when $N=1$
and there are no Markov modulation.
Then this problem has been very well understood. In this context
taking $\overline{a}=-E\xi_1>0$ and assuming that $F^{\text{\tiny \rm I}}\in\mathcal{S}$
for distribution function $F$ of $\xi_1$ from Theorem \ref{simple} we derive
classical the Pakes-Veraverbeke's Theorem:
\begin{equation}
  \label{eq:105}
  \p(M>x) \sim \frac{1}{\overline{a}}\overline{F^{\text{\tiny \rm I}}}(x)
  \qquad\text{as $x\to\infty$};
\end{equation}
see e.g. \citet{PAKES, EV}.
The intuitive idea underlying this result is the
following: the maximum $M$ will exceed a large value $x$ if the
random walk and modulating Markov chain
follow the typical behaviour specified by the law of large
numbers and stationary law ${\bs \pi}$, respectively, except that at some time $n$ a jump
occurs of size greater than $x+n\overline{a}$; this has
probability~$\oo{F}_i(x+n\overline{a})$ if Markov chain is in a state $i$;
replacing the sum over all $n$ of these
probabilities by an integral yields \eqref{eq:105}. That is,
informally, $\p(M>x)$ is equivalent (for large initial capital $x$) to the following sum
\begin{align*}
\p(M>x)&\sim \sum_{n\ge 1} \sum_{i=1}^N \pi_i\p_i(\xi_n>x+\overline{a}n, \max_{k\le n-1}S_k<x,\\&
S_{n-1}\in ((\overline{a}-\epsilon)(n-1),
(\overline{a}+\epsilon)(n-1)) \sim \sum_{n\ge 1} \sum_{i=1}^N \pi_i\p_i(\xi_n>x+\overline{a}n)\\& \sim
\sum_{i=1}^N\pi_i\int_0^\infty \overline{F}_i(x+\overline{a}t)\D t
\sim \sum_{i=1}^N\pi_ic_i\int_0^\infty \overline{F}(x+\overline{a}t)\D t
=\frac{C_S}{\overline{a}}\overline{F^{\text{\tiny \rm I}}}(x)
\end{align*}
for the random walk $S_n$ defined in \eqref{Sk} and $\epsilon >0$.
This again is the
principle of a single big jump.  See \citet{ZACH} for a short proof
of \eqref{eq:105} based on this idea (see also \citet{FZ02}).

\section{Segerdahl's approximation of finite-time ruin probability}\label{sec:seg}
Our goal in this section is generalizing the seminal result of \citet{Segerdahl} to Markov modulated risk process,
that is, to get the asymptotics of the finite time ruin probability $\mathbf{\phi}(x,t)$ introduced in \eqref{finitetimeruinprob}
when time horizon $t$ is of order $x/m+yc\sqrt{x}/m$ for
\begin{equation}\label{mcdef}
m=\tilde{\e}_{{\bs \pi}}\hat{X}_1=- \tilde{\e}_{{\bs \pi}}X_1>0\quad\text{and}\quad c^2=\widetilde{{\rm Var}}_{{\bs \pi}}\hat{X}_1=\widetilde{{\rm Var}}_{{\bs \pi}}X_1.
\end{equation}
Above $\widetilde{{\rm Var}}_{{\bs \pi}}$ denotes the variance calculated under $\tilde{\p}$
defined in \eqref{expomeasureruin} where $J_0$ starts at stationary distribution ${\bs \pi}$.
Let $\Phi_N$ be a cumulant distribution function of a standard gaussian random variable.
\begin{theorem}\label{thm:seger}
We have
\[\lim_{x\to+\infty}\frac{\phi_i(x, x/m+yc\sqrt{x}/m^{3/2})}{e^{-\gamma x}}= C_i\Phi_N(y),\]
where constant $C_i$ and the adjustment coefficient $\gamma>0$ are given in \eqref{cramerlundbergequation}.
\end{theorem}
\begin{proof}
Using the same arguments like in the proof of Theorem \ref{Crameasymptotics} we have
\begin{equation}\label{finiteruinchangemeasure}
\phi_i(x,t)=e^{-\gamma x}\tilde{\e}_{0,i}\left[e^{\gamma (\hat{X}_{\hat{\tau}_x^+}-x)}; \hat{\tau}_x^+<t\right].\end{equation}
Using Markov modulated random walk
Central Limit Theorem of \citet{Keilson} we have the following result.
\begin{lemma}
Under $\tilde{\p}$,
$\frac{\hat{X}_s-ms}{c\sqrt{s}}$ converges weakly to standard gaussian random variable $N(0,1)$ as $s\rightarrow+\infty$.
\end{lemma}
From Anscombe’s theorem it follows that
$\frac{\hat{X}_ {\hat{\tau}_x^+}-m \hat{\tau}_x^+}{c\sqrt{\hat{\tau}_x^+}}$ converges weakly to $N(0,1)$ as well.
Moreover, similarly by the Law of Large Numbers and Anscombe’s theorem
we get
\[\frac{x}{\hat{\tau}_x^+}\rightarrow m\qquad \tilde{\p}-\;{\rm a.s.\quad as}\quad x\rightarrow+\infty.\]
Further, by Renewal Theorem 28 of
 \citet{dereich2015real} the overshoot $\hat{X}_ {\hat{\tau}_x^+}-x$ converges weakly to some random variable
and hence
\[\frac{ \hat{\tau}_x^+-x/m}{c\sqrt{x}/m^{3/2}}\]
converges weakly to $N(0,1)$ as well.
The following Stam's lemma states that ruin time and deficit are asymptotically independent.
Its proof is the same as the proof of  Proposition 4.4, p.
130 of \citet{asm_ruin}.
\begin{lemma}
For bounded and continuous functions $f$ on $[0,\infty)$ and $g$ on $\mathbb{R}$ for $x>0$, we have that
\begin{equation}\label{stam}
\tilde{\e}_i\left(f(\xi(x))g\left(\frac{\hat{\tau}_x^+-\frac{x}{m}}{c\sqrt{x}/m^{3/2}}\right)\right)\sim\tilde{\e}_i f(\xi(\infty))\tilde{\e} g(N(0,1)),
\end{equation}
where $\xi(x)=\hat{X}_{\hat{\tau}_x^+}-x$.
\end{lemma}
Now, from \eqref{finiteruinchangemeasure} we get
\begin{eqnarray*}
\lefteqn{\phi_i(x,x/m+yc\sqrt{x}/m^{3/2})=e^{-\gamma x}\tilde{\e}_{0,i}\left[e^{\gamma (\hat{X}_{\hat{\tau}_x^+}-x)}; \hat{\tau}_x^+<x/m+yc\sqrt{x}/m^{3/2}\right]}\\
&&\sim e^{-\gamma x} \tilde{\e}_{i}\left[e^{\gamma\xi(\infty)}\right] \tilde{\p}\left(\frac{\hat{\tau}_{x}^+-\frac{x}{m}}{c\sqrt{x}/m^{3/2}}\le y\right)
\sim C_i e^{-\gamma x}\Phi_N(y)
\end{eqnarray*}
which completes the proof of Theorem \ref{thm:seger}.
\end{proof}
\section{H\"{o}glund's asymptotics of Markov modulated renewal function and finite-time ruin probability}\label{sec:Hog}
Recall that $E=\{1,2,\ldots,N\}$. In this section we will consider
the solution
\begin{equation}\label{renewalsolution}
U*f(x)=\sum_{n=0}^\infty \Upsilon^{*n}*f(x),\qquad x\in\R^d,\quad d\in \mathbb{N}
\end{equation}
of the renewal equation
$U-\Upsilon*U=f$
where $\Upsilon=(\Upsilon_{ij})_{\{i,j\in E\}}$ is matrix of a positive measures and $f$ is a vector of measurable functions for which series \eqref{renewalsolution} converges.
Moreover, $*$ denotes convolution in the Markov modulation set-up, that is,
\[(\Upsilon^{*k}f(x))_i=\sum_{i_1,i_2,\ldots, i_{k-1}, i_k\in E}\int f_{i_k}(x-y_1-y_2-\ldots-y_k)\Upsilon_{ii_1}(\D y_1)\Upsilon_{i_1i_2}(\D y_2)\ldots \Upsilon_{i_{k-1}i_k}(\D y_k). \]
We assume that support of each measure $\Upsilon_{ij}$ is $d$-dimensional.
Let $\lambda_{\R^d}$ denotes the Lebesgue measure in $\R^d$ and by $<\cdot,\cdot>$ we denote the scalar product in $\R^d$.
Finally, let ${\bs \pi}^\Upsilon=(\pi^\Upsilon_1, \ldots,\pi^\Upsilon_N)$ be a stationary distribution of the discrete-time
Markov chain with the transition matrix $\{\Upsilon_{ij}(\R^d)\}_{\{i,j\in E\}}$.
We denote
\begin{align*}
\Theta&=\left\{\theta\in \R^d: \int_{\R^d}|x|^2e^{<\theta,x>}\Upsilon_{ij}(\D x)<\infty,\quad i,j\in E\right\}
\end{align*}
and let for $\theta=(\theta_1, \ldots, \theta_d)\in \Theta$ the function
$\varphi(\theta)$ be a Perron-Frobenius eigenvalue of the matrix $(\int_{\R^d} e^{<\theta,x>}\Upsilon_{ij}(\D x))_{\{i,j\in E\}}$.
We also define
\begin{align*}
\varphi^\prime(\theta)=\left(\frac{\partial}{\partial \theta_1}\varphi(\theta), \ldots, \frac{\partial}{\partial \theta_d}\varphi(\theta)\right),
&\qquad \varphi^{\prime\prime}(\theta) =\left\{\frac{\partial^2}{\partial \theta_i\partial \theta_j}\varphi(\theta)\right\}_{\{i,j\in E\}},\\
\beta=\max\left\{2,\frac{d-1}{2}\right\},&
\qquad C=\varphi^\prime(\theta)(\varphi^{\prime\prime}(\theta))^{-1}\varphi^\prime(\theta){\rm det}\varphi^{\prime\prime}(\theta),\\
\end{align*}
\begin{align*}
\varpi=x(\varphi^{\prime\prime}(\theta))^{-1}\varphi^\prime(\theta)/\varphi^\prime(\theta)(\varphi^{\prime\prime}(\theta))^{-1}\varphi^\prime(\theta).
\end{align*}
We will look for the solution of the equation
\begin{equation}\label{CLequationHog}
\varphi(\theta)=1.
\end{equation}
If this solution exists, we have either $\varphi^{\prime}(\theta)\neq 0$ for all $\theta$ or
it is one-point set otherwise (compare with \citet[Lem. 1]{Hoglund1}).
We recall that  function $g$ is directly Riemann integrable
if $\int \overline{g}_h-\int \underline{g}_h$ tends to $0$ as $|h|\rightarrow 0$
for $\overline{g}_h(x)=\sup_{y\in (nh, (n+1)h]}g(x)$ and $\underline{g}_h(x)=\inf_{y\in (nh, (n+1)h]}g(x)$
for $x\in (nh, (n+1)h]$.

\begin{theorem}\label{Hoglund}
Assume that $\varphi^{\prime}(\theta)\neq 0$
and $\theta$ satisfies \eqref{CLequationHog}. Let $\varpi\ge 0$.
Suppose that $\int (|x|^2+|<c,x>|^\beta)e^{<\theta,x>}\Upsilon_{ij}(\D x)$ and $\int (1+|<c,x>|^\beta)e^{<\theta,x>}|f(x)|\lambda_{\R^d}(\D x)$ are finite for all $i,j\in E$
and that integrand in the last integral is directly Riemann integrable for some $0\neq c\in \R^d$.
In addition, we assume that $<c,\varphi(\theta)>\neq 0$ if $d=2$. Then
\begin{eqnarray*}
&&e^{<\theta,x>}(U*f(x)\1)_i=\left(2\pi \varpi\right)^{-(d-1)/2}C^{-1/2}\exp\{-(x-\varpi\varphi^\prime(\theta))(\varphi^{\prime\prime}(\theta))^{-1}(x-\varpi\varphi^\prime(\theta))\}\\
&&\qquad\times \int e^{<\theta,y>}{\bs \pi}^\Upsilon f(y)\lambda_{\R^d} (\D y)+{\rm o}((1+|<c,x>|)^{-(d-1)/2})
\end{eqnarray*}
uniformly in $x$, as $|x|\rightarrow +\infty$.
In particular, if
$\varphi^\prime(\theta)/|\varphi^\prime(\theta)|=x/|x|$ we have
\begin{equation}\label{CLasHog1}
\lim_{|x|\to+\infty}\frac{(U*f(x)\1)_i}{e^{-<\theta,x>}}=\left(2\pi \varpi\right)^{-(d-1)/2}C^{-1/2}\int e^{<\theta,y>}{\bs \pi}^\Upsilon f(y)\lambda_{\R^d} (\D y).
\end{equation}
\end{theorem}
\begin{proof}
Basically the proof is the same as the proof of \citet[Thm. 1.2 and Thm. 1.4]{Hoglund1}.
There are some minor differences though.
In any estimate one has to take all algebraic operations entrywise for matrices and
then maximum of minimum should be applied over $E$ should be applied.
The crucial difference is in the proof of equations (2.43) where Central Limit Theorem
is applied to a random walk associated with measure $\Upsilon$.
In this place one has to use \citet[p. 552]{Keilson} instead.
\end{proof}
To identify the finite time ruin probability at the beginning we follow
\citet[Prop. 3.2]{Hoglund2}.

Let $S=(S^1,S^2)= \{(S^1_n,S^2_n), n=1,2,\ldots\}$ be a (possibly killed)
Markov modulated random walk starting from $(0,0)$ whose components $S^1$ and
$S^2$ have non-negative increments.
Let $J^\Upsilon$ be a Markov chain modulating this random walk living in the state space $E=\{1,2,\ldots,N\}$.
Consider the crossing
probabilities
\begin{eqnarray}
G_{a,b,i}(x,y,i) &=& \p(N(x)<\infty, S^1_{N(x)} > x + a, S^2_{N(x)} \le x+b|J^\Upsilon_0=i),  \label{gab}\\
K_{a,b,i}(x,y,i) &=& \p(N(x)<\infty, S^1_{N(x)} > x + a, S^2_{N(x)} \ge x+b|J^\Upsilon_0=i),\nonumber
\end{eqnarray}
where $a\ge 0, b\in\R$ and
\begin{equation}\label{Nx}
N(x) = \min\{n: S_n^1  > x\}.
\end{equation}
Let ${\bs \pi}^\Upsilon$ be the stationary distribution of environmental
Markov chain $J^\Upsilon$
and
let $$V(\zeta)=\e^\zeta_{{\mathbf \pi}}[(S^2_1\e^\zeta_{{\mathbf \pi}}[S^1_1] - S_1^1
\e^\zeta_{{\mathbf \pi}}[S^2_1])^2]/ \e^{\zeta}_{{\mathbf \pi}}[S_1^1]^3$$ for $\zeta=(\eta_1,\eta_2)$ where
$\e^\zeta_{{\mathbf \pi}}$ denotes the expectation w.r.t. $$\p^\zeta_{{\mathbf \pi}}(\cdot)=
\sum_{i,j=1}^N \pi^\Upsilon_i\e[e^{<S_1, \zeta>}, S_1\in \cdot, J_1^\Upsilon=j|J^\Upsilon_0=i].$$
For our purposes it will suffice to consider
random walks that satisfy the following
assumption (the analogue of the non-lattice
assumption in one dimension):
\begin{equation*}
\tag*{(G)}
\text{The additive group spanned by the support of $F^\zeta$ contains $\R^2_+$.}
\end{equation*}
Now observe that $G_{a,b,i}(x,y)$ and $K_{a,b,i}(x,y)$ are of the form
$(\sum_{n=0}^\infty \Upsilon^{*n}*f(x)\1)_i$
with $d=2$,
\[\Upsilon(\D x, \D y)=\p(S_1^1\in \D x, S^2_1\in \D y)\]
and $f(x,y)=\Upsilon((x+a,+\infty), (-\infty, x+b])\1$
or $f(x,y)=\Upsilon((x+a,+\infty), [x+b, +\infty))\1$, respectively.
For $\theta=(\theta_1, \theta_2)\in \Theta^\circ$ let $\varphi(\theta)$ be a Perron-Frobenius eigenvalue of the matrix
$(\int_{\R^2} e^{\theta_1 x+\theta_2 y}\Upsilon_{ij}(\D x, \D y))_{\{i,j\in E\}}$.
One can easily check that all moments and direct Riemann integrability conditions
of Theorem \ref{Hoglund} are satisfied.
We will also write $f\asymp g$ if $\lim_{x,y\to\infty, x=vy+{\rm o}(y^{1/2})} f(x,y)/g(x,y) =
1$.
Applying \eqref{CLasHog1} produces the following key proposition (see Theorem 2.1 of \citet{Hoglund2}
and Cor. 2.7, p. 313 of \citet{APQ}).
\begin{prop}\label{hog}
Assume that $(G)$ holds, and that there exists a
$\zeta=(\eta_1, \eta_2)$ with $\varphi(\zeta)=1$ such that $v =
\e^\zeta_{{\mathbf \pi}}[S_1^1]/\e^\zeta_{{\mathbf \pi}}[S_1^2]$, where $\varphi$ is finite in a
neighbourhood of $\zeta$ and $(0,\eta_2)$. If $x,y$ tend to infinity
such that $x=vy+{\rm o}(y^{1/2})>0$ then it holds that
\begin{eqnarray*}
G_{a,b,i}(x,y) &\asymp & D(a,b) x^{-1/2}\te{x\eta_1+y\eta_2}
\quad \text{if\, $\eta_2>0$},\\
K_{a,b,i}(x,y) &\asymp & D(a,b) x^{-1/2}\te{x\eta_1+y\eta_2} \quad \text{if\,
$\eta_2<0$},
\end{eqnarray*}
for $a\ge 0, b\in\R$, where $D(a,b) = C(a,b) \cdot (2\pi
V(\zeta))^{-1/2} $, with $V(\zeta)>0$ and
$$
C(a,b) = \frac{1}{|\eta_2| \e^\zeta_{{\mathbf \pi}}[S_1^1]} e^{b\eta_2}\int_a^\infty
\p^\zeta_{{\mathbf \pi}}(S_1^1 \ge x)e^{\eta_1 x}\D x.
$$
\end{prop}

In the last step we follow \citet{PalmowskiPistoriusexplev}
and prove the following main result satisfied for general MAP $\hat{X}_t=-X_t$.
The result below concerns the asymptotics of the finite time ruin
probability
$$\phi_i(x)=\p(\hat{\tau}_x^+\le t)$$
when $x,t$ jointly tend to infinity
in fixed proportion. For a given proportion $v$ the rate of decay
is either equal to $\gamma v t$ or to $k^*(v)t$, where $k^*$
is the convex conjugate of the Perron-Frobenius eigenvalue $\hat{k}$ of the
matrix exponent $\hat{F}$ of the dual MAP $\hat{X}=-X$:
$$\hat{k}^*(x) = \sup_{\alpha\in\R}(\alpha x - \hat{k}(\alpha)).
$$
We restrict ourselves to the risk process in the regime switching environment satisfying the
following condition
\begin{equation}
\tag*{(H)} \text{
$\nu^{(ij)}$ defined in \eqref{nu} are non-lattice for all $i,j\in E$.}
\end{equation}
Recall that a measure is called non-lattice if its support is not
contained in a set of the form $\{a + bh, h\in\mathbb Z\}$, for
some $a,b>0$.

To present the main result we need few additional notations.

Let $\underline{L}_t$ be a local time of $X_t-\inf_{s\le t} X_s$ which in our case increases only
at epochs $\tau_k$ defined via $\tau_0=0$, $\tau_k=\inf\{t>\tau_{k-1}: X_t-\inf_{s\le t} X_s=0\}$.
Let $\underline{L}^{-1}_t$ be the inverse local time and as usual let $\underline{\bs L}^{-1}_t$ denote the vector of total times spent in different phases up to
$\underline{L}^{-1}_t$.
Observe that the ladder process $((\underline{H}_t,\underline{\bs L}^{-1}_t),J_{\underline{L}^{-1}_t})$
for $\underline{H}_t=X_{\underline{L}^{-1}_t}$
is a bivariate Markov additive process.
Thus for  $\a\ge 0$ and $\bb=(\beta_1, \ldots, \beta_N)$ with $\b_i\ge 0$ there exists matrix valued function $\underline K(\bb,\a)$ such that
\begin{equation}\label{K}
\e[\ee^{\a \underline{H}_t-\langle\bb,\underline{\bs L}^{-1}_t\rangle};J_{ \underline{L}^{-1}_t}]=\ee^{\underline K(\bb, \a)t},\end{equation}
because $\underline{L}^{-1}_0=0$.
Let  \[((\overline{\hat{H}}_t,{\overline{\hat{\bs L}}}^{-1}_t),\overline{J}_t)=
((\underline{H}_t,{\underline{{\bs L}}}^{-1}_t),\underline{J}_t)\]
with
\[\underline{J}_t=J_{{\overline{\hat{L}}}^{-1}_t}\]
be a ladder process
related with supremum of the dual risk process $\hat{X}=-X$ and $\overline{\hat{K}}(\alpha, \bb)$ is its matrix Laplace exponent.
Note that $((\underline{H}_t,{\underline{{\bs L}}}^{-1}_t),\underline{J}_t)$ is the bivariate MAP.
We denote by $\overline{{\bs \pi}}$ the stationary distribution of $\underline{J}_t$.
The assumption (H) and Vigon's formula (see \citet[Thm. 6.22 and Sec. 6..6.2]{kyprianou}) implies that
the jump measures associated to $\overline{\hat{H}}$ is non-lattice as well.
To simplify notation we will write in this section
\[(\underline{H}_t,{\underline{{\bs L}}}^{-1}_t\1)=(H_t, L_t^{-1}).\]
Let $$\Theta_{<\infty}=\{\theta\in\R: \hat{k}(\theta)=k(-\theta)<\infty\}$$
and $\Theta^\circ_{<\infty}$ be its interior.

\begin{theorem}\label{thm}
Assume that $(H)$ holds. Suppose that $0<\hat{k}^\prime(\gamma)=k^\prime(-\gamma)<\infty$ and
that there exists a $\Gamma(v)\in\Theta^\circ_{<\infty}$ such that
$\hat{k}'(\Gamma(v))=v$. If $x$ and $t$ tend to infinity such that $x=
vt+{\rm o}(t^{1/2})$ then
$$
\phi_i(x)=\p_i(\hat{\tau}_x^+\le t) \asymp \begin{cases} C_i e^{-\gamma x},
& \text{if $0< v < \hat{k}'(\gamma)$,}\\
D_v t^{-1/2}e^{-\hat{k}^*(v)t}, & \text{if $ v > \hat{k}'(\gamma)$,}
\end{cases}
$$
with $C_i$ given in Theorem \ref{Crameasymptotics} and
\begin{align*}
D_v &= \frac{1}%
{\eta_v \e_{\overline{{\bs \pi}}}[e^{\Gamma(v) H_1-\eta_v L_1^{-1}}H_1\ind{L_1^{-1}<\infty}]} \times \frac{v}{\Gamma(v)
\sqrt{2\pi\hat{k}''(\Gamma(v))}}\\&\qquad\times \overline{{\bs \pi}}\kappa(\hat{k}(\Gamma(v)),0)(q\matI+\kappa(\hat{k}(\Gamma(v)),0))^{-1}\1,
\end{align*}
where $\eta_v = \hat{k}(\Gamma(v))$.
\end{theorem}

\begin{remark}\label{uwaga1}\rm
In the case of risk process \eqref{riskMM3} with $N=1$
(there is no Marov modulation)
we can find that
$$
D_v = \frac{\Gamma(v) + \tilde\Gamma(v)}{\Gamma(v)\tilde\Gamma(v)}
\frac{1}{\sqrt{2\pi\hat{k}''(\Gamma(v))}}, \quad C_1 =
\frac{\hat{k}'(0)}{\hat{k}'(\gamma)},
$$
where $\tilde\Gamma(v) = \sup\{\theta:\hat{k}(-\theta) = \hat{k}(\G(v))\}$,
recovering formulas that can be found in \citet{Arfwedson} and
\citet{feller} respectively, for the case of a classical risk
process.
\end{remark}

\begin{remark}\rm
Note that asymptotics given in Theorem \ref{thm} works for different range
of time horizons than Segerdahl's approximation identified in Theorem
\ref{thm:seger}.
\end{remark}

\begin{proof}
At the beginning we will use the exponential change of measure \eqref{expomeasureruin} defining $\tilde{\p}_i$.
In the case $0<v<\hat{k}'(\gamma)$, the asymptotics in Theorem
\ref{thm} are a consequence of the law of large numbers. To see
why this is the case, note that $e^{\gamma x}\p_i(\hat{\tau}_x^+\le t) =
e^{\gamma x}\p_i(\hat{\tau}_x^+ < \infty) - e^{\gamma x}\p_i(t<\hat{\tau}_x^+ < \infty)$,
where the first term tends to $C_i$ in view of Theorem \ref{Crameasymptotics},
while for the second term the Markov property
imply that
\begin{eqnarray*}
\lefteqn{e^{\gamma x}\p_i(t<\hat{\tau}_x^+ < \infty)}\\&&\le\sum_{j\in E}\int_{-\infty}^x
\p_i(\hat{\tau}_x^+>t, \hat{X}_t\in \D y)e^{\gamma y}e^{\gamma (x-y)}\p_j(\hat{\tau}_{x-y}^+<\infty)\\
&&\le \int_{-\infty}^x \p_i(\hat{X}_t\in \D y)e^{\gamma
y}=\tilde{\p}_i(\hat{X}_t\le x),
\end{eqnarray*}
which tends to $0$ as $t$ tends to infinity in view of the law of
large numbers since $\tilde{\e}[\hat{X}_t]=t\hat{k}'(\gamma)>x=vt+{\rm o}(t^{1/2})$.
We will now consider the case of $v>\hat{k}'(\gamma)$.

We start from construction of the embedded Markov modulated random walk
at Poisson epochs.
Denote by $e_1, e_2, \ldots$ a sequence of independent identically exponentially distributed with
parameter $q$ random variables and by $\s_n=\sum_{i=1}^n e_i$, with
$\s_0=0$, the corresponding partial sums. Consider the
two-dimensional (killed) Markov modulated random walk $\{(S_n^1,S_n^2), n=1,2\ldots\}$
starting from $(0,0)$ with step-sizes distributed according to
$$
\Upsilon^{(q)}(\D x, \D t) =
\p(S_1^1\in \D x, S_1^2\in \D t)= \p(
H_{\s_{1}}\in \D x, L^{-1}_{\s_{1}}\in\D t),
$$
and write $G^{(q)}$ for the corresponding crossing probability
$$
G^{(q)}(x,t) = G_{0,0,i}(x,t) = \p_i(N(x)<\infty, S^2_{N(x)}
\le t)
$$
for $N(x)$ defined in \eqref{Nx} and $G_{a,b,i}(x,y)$ defined in \eqref{gab}.
Note that
\begin{equation}\label{phikappalink}
\iint e^{-ux-vt} \p(S_1^1\in \D x, S_1^2\in \D t)  = q(q\matI
- \kappa(v,u))^{-1},
\end{equation}
where
\begin{equation}\label{kappa}
\kappa(v,u)=\overline{\hat{K}}(v\1^T, u),
\end{equation}
and $\overline{\hat{K}}$ is a matrix Laplace exponent of
a ladder process related with supremum of the dual risk
$\hat{X}=-X$ defined formally below \eqref{K}.

The key step in the proof is to derive bounds for $\p(\hat{\tau}_x^+\le
t)$ in terms of crossing probabilities involving the random walk
$(S^1,S^2)$.
\begin{lemma}\label{lem:est} Let $M,q>0$. For $x,t>0$ it holds that
\begin{equation}\label{star}
G^{(q)}(x,t) \le \p_i(\hat{\tau}_x^+\le t) \le G^{(q)}(x,t+M)/\min_{i\in E} h_i(0-, M),
\end{equation}
where $h(0-,M) = \lim_{x\uparrow 0} h(x,M)$, with $h(x,t) =
\p(H_{\s_1} > x, L^{-1}_{\s_1} \le t)\1$.
\end{lemma}

\begin{proof}
Let $T(x) = \inf\{t\ge0: H_t > x\}$ and note that
$\hat{\tau}_x^+ = L^{-1}_{T(x)}$. By applying the Markov property it
follows that
\begin{eqnarray}
 \p_i(\hat{\tau}_x^+\le t)
 &=& \p_i(T(x) < \infty, L^{-1}_{T(x)} \le t)\\
  &=& \sum_{n=1}^\infty \p(\s_{n-1} \le T(x) < \sigma_n, L^{-1}_{T(x)} \le
  t)\nonumber\\
  &=& \sum_{n=1}^\infty \p(H_{\s_{n-1}}\le x, H_{\s_n} > x,
L^{-1}_{T(x)} \le
  t)\nonumber\\
  &=& \sum_{n=1}^\infty\int  \p_i(H_{\s_{n-1}}\in \D y,
  L^{-1}_{\s_{n-1}}\in\D s) \\&& \quad \quad \times\ \p(H_{\s_1} > x - y, L^{-1}_{T(x-y)} \le
  t-s)\nonumber\\
  &=& \left(\sum_{n=0}^\infty (\Upsilon^{(q)})^{\star n} \star f(x,t)\right)_i = ((U \star
  f)(x,t))_i,\label{eq:u*f}
\end{eqnarray}
where $U=\sum_{n=0}^\infty (\Upsilon^{(q)})^{\star n}$, $f(x,t) = \p(H_{\s_1}
> x, L^{-1}_{T(x)} \le t)\1$ and $\star$ denotes convolution.
Following a similar reasoning it can be
checked that \begin{equation}\label{eq:u*h}G^{(q)}(x,t) = (U\star
h(x,t)\1)_i.\end{equation} In view of \eqref{eq:u*f} and
\eqref{eq:u*h}, the lower bound in \eqref{star} follows since $$f(x,t) \ge h(x,t),$$ taking note of the fact that $H_{\s_1}> x$
precisely if $T(x) < \s_1$,
while the upper bound in \eqref{star} follows by observing that
for fixed $M>0$,
\begin{eqnarray*}
\nonumber  h(x,t+M) &\ge& \p(H_{\s_1} > x, L^{-1}_{T(x)}
\le t, L^{-1}_{\s_1} - L^{-1}_{T(x)} \le M)\1\\
\nonumber &=& \p(H_{\s_1} > x, L^{-1}_{T(x)} \le t) \p(L^{-1}_{\s_1} \le
M)\1\\ &\ge& f(x,t)\min_{i\in E} h_i(0-,M),
\end{eqnarray*}
where we used the strong Markov property of $L^{-1}$ and the lack
of memory property of $\s_1$. \end{proof}

\medskip

\noindent
Applying H\"oglund's asymptotics in Proposition \ref{hog} yields
the following result.

\begin{lemma}\label{hog2}
Let the assumptions of Proposition \ref{hog} hold true. If
$x,t\to\infty$ such that for $v > \hat{k}'(\gamma)$ we have $x=vt+{\rm o}(t^{1/2})$ then
$$
G^{(q)}(x,t+M) \sim D_{q,M}t^{-1/2}e^{-\hat{k}^*(v)t}, \quad M\ge 0,
$$
where $D_{q,M} = \frac{v}{\sqrt{2\pi\hat{k}''(\Gamma(v))}} C_{q,M}$ with
$$
C_{q,M} =
\frac{q e^{\hat{k}(\Gamma(v))M}}{c_v\hat{k}(\Gamma(v))\Gamma(v)}\overline{{\bs \pi}}\kappa(\hat{k}(\Gamma(v)),0)(q\matI+\kappa(\hat{k}(\Gamma(v)),0))^{-1}\1,
$$
where $c_v=\e_{\overline{{\bs \pi}}}[e^{\Gamma(v) H_{1} - \hat{k}(\G(v))L^{-1}_{1}}H_{1} \ind{L^{-1}_{1}<\infty}]$
and $\kappa$ is defined in \eqref{kappa}.
\end{lemma}

For $u>\gamma$ and $u\in\Theta^\circ_{<\infty}$ we denote by $\p^{(u)}$ the measure
\begin{equation}\label{expomeasureruinb}
\frac{\D \p^{(u)}_{|\mathcal{F}_t}}{\D\p_{|\mathcal{F}_t}}=e^{u (X_t-x)-k(u)t}
\frac{h_{J_t}(u)}{h_{J_0}(u)}.\end{equation}
Let $\e^{(u)}$ be the expectation with respect of $\p^{(u)}$.
Lemma \ref{hog2} is a consequence of the following auxiliary
identities given in \citet{PalmowskiPistoriusexplev}:
\begin{eqnarray}
\label{eq:a}\varphi(z,-u)&=&1 \quad \text{iff  $\kappa(z,-u)=0$ iff
$\hat{k}(u)=z$},
\\
\label{eq:b} \hat{k}'(u) &=&
\e^{(u)}_{\overline{{\bs \pi}}}[\hat{X}_1]=\e^{(u)}_{\overline{{\bs \pi}}}[H_{\s_1}]\cdot(\e^{(u)}_{\overline{{\bs \pi}}}[L^{-1}_{\s_1}])^{-1},\\
\nonumber  \hat{k}''(u)& =& \e^{(u)}_{\overline{{\bs \pi}}}[(H_{\s_1} -
\hat{k}'(u)L^{-1}_{\s_1})^2]\cdot(\e^{(u)}_{\overline{{\bs \pi}}}[L^{-1}_{\s_1}])^{-1},\\
\label{eq:c} &=&\hat{k}'(u)\e^{(u)}_{\overline{{\bs \pi}}}[(H_{\s_1} -
\hat{k}'(u)L^{-1}_{\s_1})^2]\cdot(\e^{(u)}_{\overline{{\bs \pi}}}[H_{\s_1}])^{-1},\\
\label{eq:d} \hat{k}^*(v) &=& v\Gamma(v) - \hat{k}(\Gamma(v))\quad
\text{for $v>0$ with $\Gamma(v)\in\Theta^\circ_{<\infty}$.}
\end{eqnarray}
In particular, the first identity follows from \eqref{phikappalink}, \eqref{kappa} and
Wiener-Hopf factorization given in Theorem 26 of \citet{dereich2015real}.

\begin{proof}[Proof of Lemma \ref{hog2}] The proof follows by an
application of Prop. \ref{hog} to $G^{(q)}(x,t+M)$ with
$$(S_1^1, S_1^2)
= (H_{\s_1},L^{-1}_{\s_1})\quad \text{ and }\quad
\zeta=(-\Gamma(v),\eta_v).$$ Note that, by \eqref{eq:a} with
$u=\Gamma(v)$, $\varphi(\zeta)=1$, and that $\eta_v = \hat{k}(\G(v))>0$
if $v > \hat{k}'(\gamma)$. For this choice of the parameters,
$\e^\zeta_{{\bs \pi}}[S_1^1]=\e^{(\Gamma(v))}_{\overline{{\bs \pi}}}[H_{\s_1}]=c_v/q$, and Equations
\eqref{eq:b},\eqref{eq:c} \eqref{eq:d} imply that  $\xi x+\eta
t=-\hat{k}^*(v)t$ and
$$V(\zeta)=
\hat{k}''(\Gamma(v))/\hat{k}'(\Gamma(v))=\hat{k}''(\Gamma(v))/v.$$ To
complete the proof we are left to verify
 the form of the constants. The calculation of the
$C_{q,M}=C(0,0)e^{\eta_v M}$ goes as follows:
\begin{eqnarray*}
C_{q,M} &=&
\frac{qe^{\hat{k}(\Gamma(v))M}}{\hat{k}(\Gamma(v))c_v}\left(\int_0^\infty
e^{-\Gamma(v)x}\e_{\overline{{\bs \pi}}}[e^{\Gamma(v) H_{\s_1} - \hat{k}(\G(v))L^{-1}_{\s_1}}
\ind{(x\le H_{\s_1} < \infty)}]\D x \right)\\
&=& \frac{qe^{\hat{k}(\Gamma(v))M}}{\hat{k}(\Gamma(v))\Gamma(v)c_v}
\left(1 - \e_{\overline{{\bs \pi}}}[e^{-\hat{k}(\Gamma(v))L^{-1}_{\s_1}}\ind{(L^{-1}_{\s_1} < \infty)}]\right)\\
&=& \frac{qe^{\hat{k}(\Gamma(v))M}}{\hat{k}(\Gamma(v))\Gamma(v)c_v}\left(1 -
q\overline{{\bs \pi}}(q\matI+\kappa(\hat{k}(\Gamma(v)),0))^{-1}\1\right).\\
&=& \frac{qe^{\hat{k}(\Gamma(v))M}}{\hat{k}(\Gamma(v))\Gamma(v)c_v}
\overline{{\bs \pi}}\kappa(\hat{k}(\Gamma(v)),0)(q\matI+\kappa(\hat{k}(\Gamma(v)),0))^{-1}\1,
\end{eqnarray*}
Combining all
results completes the proof.\end{proof}

Now we can complete the proof of Theorem \ref{thm}.
Writing $l(t,x) = t^{1/2}e^{\hat{k}^*(v)t}\p(\hat{\tau}_x^+\le t)$, Lemmas
\ref{lem:est} and \ref{hog2} imply that
\begin{eqnarray*}
s &=& \limsup_{x,t\to\infty, x=tv+{\rm o}(t^{1/2})} l(t,x) \le D_{q,M}/\min_{i\in E} h_i(0-,M),\\
i &=& \liminf_{x,t\to\infty, x=tv+{\rm o}(t^{1/2})} l(t,x) \ge D_{q,0}.
\end{eqnarray*}
By definition of $h$ and $D_{q,M}$ it directly follows that, as
$q\to\infty$,
$$D_{q,0} \to D_v,\ D_{q,M} \to D_v e^{\hat{k}(\G(v))M}\text{ and }
h(0-,M)=\p(L^{-1}_{\sigma_1}\le M)\1\to \1.$$ Letting $M\downarrow 0$
yields that $s=i = D_v$, and the proof is complete.
\end{proof}

\section{Parisian ruin}\label{sec:Par}
In this section we follow \citet{CzarnaPalmowski, Loeffenetalpar, DassWu1}
considering so-called Parisian ruin probability,
that occurs if the risk process $X$ defined in \eqref{riskMM3} stays below zero for a longer period than a fixed $\zeta>0$.
Formally, we define
Parisian time of ruin by:
$$\tau^{\zeta}=\inf\{t>0: t-\sup\{s<t: X_s \ge 0\}\ge \zeta, X_t<0 \}$$
and Parisian ruin probability is then given by:
$$\p_{x,i}(\tau^\zeta<\infty).$$
The case $\zeta=0$ corresponds to classical ruin problem and we do not deal with this case in this section.
The name for this problem is borrowed from the Parisian option. Depending on type of such option the prices
are activated or canceled if
underlying asset stays above or below barrier long enough in a row.
It is a common belief that Parisian ruin probability is a better measure of risk
than classical ruin probability in many situations
giving possibility for insurance company to get solvency.
We consider here the fixed delay $\zeta$. In other
papers the deterministic and fixed delay $\zeta$ is replaced by an
independent exponential random variable; see e.g. \citet{Land2, EriketalGerrberpar}.
In this case, as pointed in \citet{Jevsstriking, Gordonetal}, the ruin probability
is closely related with Poisson observed ruin probability.
The paper that deal with this probability in the context of MAP risk processes is \citet{Chineeseparisian}.

We give in next theorem different main representation of the Parisian survival probability though.
Recall that by
\[\tau_x^+=\inf\{t\ge 0:X_t>x\}\]
we denote the first passage time over a level $x\ge 0$
and observe that $X_{\tau_x^+}=x$ a.s., because of absence of positive jumps.
The Markov additive property applied at $\tau_x^+$ implies that $J_{\tau_x^+}$, the phase observed at the first passage times, is a Markov chain indexed by the level~$x\ge 0$.
Hence there is an identity:
\begin{align}
\label{eq:Gdef_cont}
 \e[e^{-q\tau_x^+}, J_{\tau_{x}^+}]=e^{G^{(q)} x}, \quad x\ge 0
\end{align}
for some \emph{transition rate} matrix~$G^{(q)}$ which is of size $N\times N$.
Moreover, by \citet{ivanovs_palmowski} the matrix $G^{(q)}$ is a \emph{right solution} of
$$F(-G^{(q)})=q\matI.$$
\begin{theorem}\label{ThmMainpar}
Parisian survival probability for a MAP risk process equals:
\begin{eqnarray} \label{PRE}
\lefteqn{\p_{x,i}(\tau^\zeta=+\infty)=
\p_{x,i}(\tau_0^-=+\infty)}\\\nonumber&&+
\int_0^\infty \p_{x,i}(\tau^{-}_0<\infty, -X_{\tau^{-}_0}\in \D z)\p(\tau_z^+\le \zeta)\p(\tau^\zeta=+\infty),
\end{eqnarray}
where $\p_{x,i}(\tau_0^-=+\infty)=\overline{\phi}_i(x)=(W(x)W(\infty)^{-1}\1)_i$ by \eqref{ruintimeprob},
and the vector $\p(\tau^\zeta=+\infty)=(\p_i(\tau^\zeta=+\infty))_{\{i\in E\}}$ solves the following system of equations
\begin{eqnarray} \label{PRE2}
&&\qquad\p_{i}(\tau^\zeta=+\infty)=
\p_{i}(\tau_0^-=+\infty)\\\nonumber&&+
\sum_{j,k\in E}\int_0^\infty \p_{i}(\tau^{-}_0<\infty, -X_{\tau^{-}_0}\in \D z, J_{\tau_0^-}=k)\p_k(\tau_z^+\le \zeta, J_{\tau_z^+}=j)\p_j(\tau^\zeta=+\infty).
\end{eqnarray}
Moreover,
\begin{eqnarray}\label{ltxuj}
\lefteqn{\int_0^{\infty}e^{-\theta s}\,\D s\int_0^\infty
\p_{x,i}(\tau^{-}_0<\infty, -X_{\tau^{-}_0}\in \D z) \p(\tau_z^+\le s)}
\\ \label{ltx}
&& =\frac{1}{\theta}\int_0^\infty \int_0^\infty u^{(0)}(x,z)\; \nu (z+\D y)\D ze^{G^{(\theta)}y},
\end{eqnarray}
where
\[u^{(0)}(x,z)=W(x)e^{R z}-W(x-z).\]
\end{theorem}
\begin{proof}
On the event $\{\tau^\zeta=\infty\}$ we decompose possible trajectory that goes below zero into two parts. The first one starts at the undershoot of $0$ of size, say, $-z<0$ visiting zero
in continuous way because of the spectral negativity of $X$ in a shorter period than $\zeta$. The second part
starts at $0$ after this excursion below $0$.
Using the strong Markov property it will produce:
\begin{eqnarray*}
\lefteqn{\p_{x,i}(\tau^\zeta=\infty)=
\p_{x,i}(\tau_0^-=\infty)}\\\nonumber&&+
\int_0^\infty \p_{x,i}(\tau^{-}_0<\infty, -X_{\tau^{-}_0}\in \D z)\p(\tau_z^+\le \zeta)
\p(\tau^\zeta=\infty).
\end{eqnarray*}
This justifies the equation (\ref{PRE}). System of equations \eqref{PRE2}
follows straightforward from (\ref{PRE}) by taking $x=0$ there.
Finally, note that by \eqref{eq:Gdef_cont},
\begin{equation}\label{LapXeq}\int_0^{\infty}e^{-\theta s}\p(\tau_z^+\le s)\,\D s = \frac{1}{\theta}\e_{x,i}\left(e^{-\theta\tau_z^+},\tau_z^+<\infty\right)
=\frac{1}{\theta}e^{G^{(\theta)}z}.
\end{equation}
Further, from the compensation formula \eqref{compformula} and \eqref{compensationrepresentation}
we have
\begin{eqnarray}\lefteqn{\int_0^{\infty}e^{-\theta s}\,\D s\int_0^\infty
\p_{x,i}(\tau^{-}_0<\infty, -X_{\tau^{-}_0}\in \D z)\p(\tau_z^+\le s)}\nonumber\\&&=
\frac{1}{\theta}\int_0^\infty \p_{x,i}(\tau^{-}_0<\infty, -X_{\tau^{-}_0}\in \D y)e^{G^{(\theta)}y}
\nonumber\\&&=
\frac{1}{\theta}\int_0^\infty \int_0^\infty u^{(0)}(x,z)\; \nu (z+\D y)\D ze^{G^{(\theta)}y}
\label{waznarep}
\end{eqnarray}
which completes the proof.
\end{proof}

We will derive now the Cram\'er's estimate of the Parisian ruin probability.
\begin{theorem}\label{Crameasymptoticspar}
Under Cram\'{e}r condition \eqref{cramerlundbergequation},
\[\lim_{x\to+\infty}\frac{\p_{x,i}(\tau^\zeta<+\infty)}{e^{-\gamma x}}=C^\zeta\]
for some finite constant $C^\zeta>0$.
\end{theorem}
 \begin{proof}
We follow the same idea like in the proof of Theorem \ref{Crameasymptotics}.
That is, from \eqref{PRE}
\begin{eqnarray} \label{PREx}
&&\qquad\p_{x,i}(\tau^\zeta<+\infty)=
\p_{x,i}(\tau_0^-<+\infty)\\\nonumber&&-
\int_0^\infty \p_{x,i}(\tau^{-}_0<\infty, -X_{\tau^{-}_0}\in \D z)\p(\tau_z^+\le \zeta)\p(\tau^\zeta=+\infty)\\
&&\qquad=e^{-\gamma x}\tilde{\e}_{0,i}\left[e^{\gamma (\hat{X}_{\hat{\tau}_x^+}-x)}\frac{h_{i}(\gamma)}{h_{J_{\hat{\tau}_x^+}}(\gamma)}\right]
\\\nonumber&&-
e^{-\gamma x}\int_0^\infty
\tilde{\e}_{0,i}\left[e^{\gamma (\hat{X}_{\hat{\tau}_x^+}-x)}\frac{h_{i}(\gamma)}{h_{J_{\hat{\tau}_x^+}}(\gamma)}, (\hat{X}_{\hat{\tau}_x^+}-x)\in \D z\right]
\p(\tau_z^+\le \zeta)\p(\tau^\zeta=+\infty).
\end{eqnarray}
Using Renewal Theorem 28 of
 \citet{dereich2015real} and Tonelli theorem complete the proof.
\end{proof}

We will  move now to MAP risk process considered in Section
\ref{sec:subexpnentialasruinprob} in which
$X$ is a Markov modulated drift minus compound Poisson process
and claim size are subexponential.

\begin{theorem}\label{Thmconveq}
Under assumptions of Theorem \ref{simple} we have
\begin{displaymath}
  \lim_{x\to\infty}\frac{\p_{y,i}(\tau^\zeta<+\infty)}{\oo{F^{\text{\tiny \rm I}}}(x)}
  = \frac{C}{\overline{a}}
\end{displaymath}
for $C$ and $\overline{a}$ given in \eqref{eq:3} and \eqref{driftMMRW}, respectively.
\end{theorem}
\begin{proof}
Note that from definition of Parisian ruin time if follows that
\[\p_{x,i}(\tau^\zeta<+\infty)\le \p_{x,i}(\tau_0^-<+\infty).\]
Moreover, conditioned that risk process $X$ got ruined, the deficit cannot be larger that
$p=\max_{i\in E}p_i \zeta$, otherwise it will not manage to return to zero.
Thus by Theorem \ref{ThmMainpar}
\begin{eqnarray*}
\lefteqn{\p_{x,i}(\tau^\zeta=+\infty)=
\p_{x,i}(\tau_0^-=+\infty)}\\\nonumber&&+
\int_0^p \p_{x,i}(\tau^{-}_0<\infty, -X_{\tau^{-}_0}\in \D z)\p(\tau_z^+\le \zeta)\p(\tau^\zeta=+\infty)
\end{eqnarray*}
and hence
\begin{eqnarray*}
\lefteqn{\p_{x,i}(\tau^\zeta<+\infty)\ge
\p_{x,i}(\tau_0^-<+\infty)}\\\nonumber&&-
\p_{x,i}(\tau^{-}_0<\infty, -X_{\tau^{-}_0}\le p).
\end{eqnarray*}
Further,
\begin{align*}
\p_{x,i}(\tau^{-}_0<\infty, -X_{\tau^{-}_0}\le p)
&= \p_{x,i}(\tau^{-}_0<\infty)-\p_{x,i}(\tau^{-}_0<\infty, -X_{\tau^{-}_0}> p)\\&
\le\p_i(M>x)\\&-\sum_{n=1}^\infty\p_i(\max_{k\le n-1}S_k<x, S_{n-1}\ge -a_{n-1}, \xi_n>x+a_{n-1}+p),
\end{align*}
where $S_k$ is defined in \eqref{Sk} and $a_n=a_0+(\overline{a}+\epsilon)n$ for some $a_0, \epsilon>0$.
Now using principle of one big jump (see inequalities (110)-(114) of  \citet{FKZ})
and fact that any subexponential distribution is long-tailed
one get that
\[\sum_{n=1}^\infty\p_i(\max_{k\le n-1}S_k<x, S_{n-1}\ge -a_{n-1}, \xi_n>x+a_{n-1}+p)
\ge (1+{\rm o}(1))\frac{C}{\overline{a}}\oo{F^{\text{\tiny \rm I}}}(x).\]
Thus
\[\p_{x,i}(\tau^{-}_0<\infty, -X_{\tau^{-}_0}\le p)={\rm o}(\oo{F^{\text{\tiny \rm I}}}(x))\]
and this gives the assertion of the theorem.
\end{proof}

\bibliographystyle{plainnat}

\begin{thebibliography}{47}
\providecommand{\natexlab}[1]{#1}
\providecommand{\url}[1]{\texttt{#1}}
\expandafter\ifx\csname urlstyle\endcsname\relax
  \providecommand{\doi}[1]{doi: #1}\else
  \providecommand{\doi}{doi: \begingroup \urlstyle{rm}\Url}\fi

\bibitem[Ivanovs (2017)]{Jevsstriking}
H.~Albrecher and J.~Ivanovs.
\newblock Strikingly simple identities relating exit problems for L\'evy
  processes under continuous and poisson observations.
\newblock \emph{Stoch. Process. Appl.}, 127, 643--656, 2017.

\bibitem[Alsmeyer(1994)]{Alsmeyer}
G.~Alsmeyer.
\newblock On the Markov renewal theorem.
\newblock \emph{Stochastic Process. Appl.}, 50(1), 37--56,
  1994.

\bibitem[Arfwedson (1955)]{Arfwedson}
G.~Arfwedson.
\newblock Research in collective risk theory.
\newblock \emph{Skand. Aktuarietidskr.}, 38, 53--100, 1955.

\bibitem[Asmussen (1989)]{asmu}
S.~Asmussen.
\newblock Risk theory in a {M}arkovian environment.
\newblock \emph{Scandinavian Actuarial Journal}, 2, 69--100, 1989.

\bibitem[Asmussen (2003)]{APQ}
S.~Asmussen.
\newblock \emph{Applied probability and queues.},
\newblock Springer-Verlag, New York, second edition, 2003.

\bibitem[Asmussen and Albrecher (2010)]{asm_ruin}
S.~Asmussen and H.~Albrecher.
\newblock \emph{Ruin probabilities}.
\newblock Advanced Series on Statistical Science \& Applied Probability,
  World Scientific Publishing, second edition,
  2010.

\bibitem[Athreya et~al. (1978)]{Athreyarenewal}
K.B.~Athreya, D.~McDonald, and P.~Ney.
\newblock Limit theorems for semi-Markov processes and renewal theory for
  Markov chains.
\newblock \emph{Ann. Probab.}, 6(5), 788--797, 1978.

\bibitem[Badescu and Landriault (2009)]{BadescuDavid}
A.L.~Badescu and D.~Landriault.
\newblock Applications of fluid flow matrix analytic methods in ruin theory — a review.
\newblock \emph{RACSAM Rev. R. Acad. Cien. Serie A. Mat.}, 103(2), 353--372, 2009.

\bibitem[Baurdoux et~al. (2016)]{EriketalGerrberpar}
E.J.~Baurdoux, J.C.~Pardo, J.L.~P\'erez, and J.-F.~Renaud.
\newblock Gerber–shiu distribution at Parisian ruin for L\'evy insurance risk
  processes.
\newblock \emph{J. Appl. Probab.}, 53(2), 572--584, 2016.

\bibitem[Bertoin and Doney (1994)]{BerDonCramer}
J.~Bertoin and R.~Doney.
\newblock Cram\'er's estimate for L\'evy processes.
\newblock \emph{Stat. Prob. Lett.}, 21, 363--365, 1994.

\bibitem[Bin et~al. (2018)]{Gordonetal}
L.~Bin, G.E.~Willmot, and J.T.Y.~Wong.
\newblock A temporal approach to the Parisian risk model.
\newblock \emph{J. Appl. Probab.}, 55, 302--317, 2018.

\bibitem[\c{C}inlar (1972)]{Cinlar1}
E.~\c{C}inlar.
\newblock Markov additive processes. I.
\newblock \emph{Z. Wahrscheinlichkeitstheorie verw. Geb.}, 24, 85--93,
  1972.

\bibitem[Cheung and Landriault (2009a)]{CheungDavid}
E.~Cheung and D.~Landriault.
\newblock Perturbed map risk models with dividend barrier strategies.
\newblock \emph{J. Appl. Probab.}, 46, 521--541, 2009.

\bibitem[Cheung and Landriault (2009b)]{Dawid}
E.~Cheung and D.~Landriault.
\newblock Analysis of a generalized penalty function in a semi-Markovian risk model.
\newblock \emph{North American Actuarial Journal}, 13(4), 497--513, 2009.

\bibitem[Chistyakov (1964)]{CHIST}
V.P.~Chistyakov.
\newblock A theorem on sums of independent positive random variables and its
  applications to branching random processes.
\newblock \emph{Theory Probab. Appl.}, 9, 640--648, 1964.

\bibitem[Czarna and Palmowski (2011)]{CzarnaPalmowski}
I.~Czarna and Z.~Palmowski.
\newblock Ruin probability with parisian delay for a spectrally negative L\'evy
  risk process.
\newblock \emph{J. Appl. Probab.}, 48(4), 984--1002, 2011.

\bibitem[Dassios and Wu (2008)]{DassWu1}
A.~Dassios and S.~Wu.
\newblock Parisian ruin with exponential claims.
\newblock \emph{Unpublished manuscript}, Available at
  http://stats.lse.ac.uk/angelos/, 2008.

\bibitem[Dereich et~al. (2017)]{dereich2015real}
S.~Dereich, L.~D\"{o}ring, and A.~Kyprianou.
\newblock Real self-similar processes started from the origin.
\newblock \emph{Ann. Probab.}, 45(3), 1952--2003, 2017.

\bibitem[Embrechts and Veraverbeke (1982)]{EV}
P.~Embrechts and N.~Veraverbeke.
\newblock Estimates for the probability of ruin with special emphasis on the
  possibility of large claims.
\newblock \emph{Insurance Math. Econom.}, 1, 55--72, 1982.

\bibitem[Embrechts et~al. (1997)]{EKM}
P.~Embrechts, C.~Kl\"uppelberg, and T.~Mikosch.
\newblock \emph{Modelling Extremal Events}.
\newblock Springer-Verlag, 1997.

\bibitem[Feller (1966)]{feller}
W.~Feller.
\newblock \emph{An Introduction to Probability Theory and its Applications},
  volume~II.
\newblock John Wiley \& Sons, 1966.

\bibitem[Foss and Zachary (2002)]{FZ02}
S.~Foss and S.~Zachary.
\newblock Asymptotics for the maximum of a modulated random walk with
  heavy-tailed increments.
\newblock \emph{In: Analytic Methods in Applied Probability (in memory of
  Fridrih Karpelevich)}, 207, 37--52, 2002.

\bibitem[Foss et~al. (2007)]{FKZ}
S.~Foss, T.~Konstantopoulos, and S.~Zachary.
\newblock Discrete and continuous time modulated random walks with heavy-tailed
  increments.
\newblock \emph{J. Theor. Probab.}, 20, 581--612, 2007.

\bibitem[Foss et~al. (2013)]{FDZ}
S.~Foss, D.~Korshunov, and S.~Zachary.
\newblock \emph{An Introduction to Heavy-Tailed and Subexponential
  Distributions}.
\newblock Springer-Verlag, 2013.

\bibitem[H\"{o}glund (1988)]{Hoglund1}
T.~Hoglund.
\newblock A mutidimensional renewal theorem.
\newblock \emph{Bull. Sc. math., $2^{e}$ seria}, 112, 111--138, 1988.

\bibitem[H\"{o}glund (1990)]{Hoglund2}
T.~Hoglund.
\newblock An asymptotic expression for the probability of ruin within finite
  time.
\newblock \emph{Ann. Probab.}, 18(1), 378--389, 1990.


\bibitem[Ivanovs (2014)]{ivanovs_potential}
J.~Ivanovs.
\newblock Potential measures of one-sided Markov additive processes with
  reflecting and terminating barriers.
\newblock \emph{J. Appl. Probab.}, 51(4),  1154--1170, 2014.

\bibitem[Ivanovs and Palmowski(2012)]{ivanovs_palmowski}
J.~Ivanovs and Z.~Palmowski.
\newblock Occupation densities in solving exit problems for {M}arkov additive
  processes and their reflections.
\newblock \emph{Stochastic Process. Appl.}, 122(9), 3342--3360, 2012.

\bibitem[Jacobsen(2005)]{Jacobsen}
M.~Jacobsen.
\newblock The time to ruin for a class of markov additive risk process with
  two-sided jumps.
\newblock \emph{Adv. Appl. Probab.}, 37(4), 963--992, 2005.

\bibitem[Jacod and Shiryaev (2003)]{JacodShiryaev}
J.~Jacod and A.~Shiryaev.
\newblock \emph{Limit theorems for stochastic processes}.
\newblock Springer-Verlag, 2003.

\bibitem[Keilson and Wishart (1964)]{Keilson}
J.~Keilson and D.M.G~Wishart.
\newblock A central limit theorem for processes defined on a finite markov
  chain.
\newblock \emph{Proc. Camb. Phil. Soc.}, 60, 547--567, 1964.

\bibitem[Kesten (1974)]{kesten2}
H.~Kesten.
\newblock Renewal theory for functionals of a Markov chain with general state
  space.
\newblock \emph{Ann. Probab.}, 2(3), 355--386, 1974.

\bibitem[Kyprianou (2013)]{AndreasGerberShiu}
A.~Kyprianou.
\newblock \emph{Gerber–Shiu Risk Theory}.
\newblock Springer, 2013.

\bibitem[Kyprianou and Palmowski (2005)]{kyprianoupalmowski}
A.~Kyprianou and Z.~Palmowski.
\newblock A martingale review of some fluctuation theory for spectrally
  negative {L}\'evy processes.
\newblock In \emph{S\' eminaire de Probabilit\' e {XXXVIII}}, 16--29.
  Springer, 2005.

\bibitem[Kyprianou (2014)]{kyprianou}
A.~Kyprianou.
\newblock \emph{Introductory Lectures on Fluctuations of L\'evy Processes with
  Applications}.
\newblock Springer-Verlag, second edition, 2014.

\bibitem[Kuznetsov and Rivero(2013)]{KKR}
A.E.~Kyprianou A.~Kuznetsov and V.~Rivero.
\newblock The theory of scale functions for spectrally negative L\'{e}vy
  processes.
\newblock L\'evy Matters II, Springer Lecture Notes in Mathematics, 97--186,
2013.

\bibitem[Lalley (1984)]{lalley}
S.P.~Lalley.
\newblock Conditional Markov renewal theory I. Finite and denumerable state
  space.
\newblock \emph{Ann. Probab.}, 12(4), 1113--1148, 1984.

\bibitem[Landriault et~al. (2014)]{Land2}
D.~Landriault, J.-F. Renaud, and X.~Zhou.
\newblock Insurance risk models with Parisian implementation delays.
\newblock \emph{Methodology and Computing in Applied Probability}, 16, 583--607, 2014.

\bibitem[Loeffen et~al. (2013)]{Loeffenetalpar}
R.~Loeffen, I.~Czarna, and Z.~Palmowski.
\newblock Parisian ruin probability for spectrally negative L\'evy processes.
\newblock \emph{Bernoulli}, 19(2), 599--609, 2013.

\bibitem[Ng and Yang (2006)]{NgYang}
A.C.Y. Ng and H.~Yang.
\newblock On the joint distribution of surplus before and after ruin under a
  markovian regime switching model.
\newblock \emph{Stochastic Process. Appl.}, 116:\penalty0 244--266, 2006.

\bibitem[Pakes (1975)]{PAKES}
A.~Pakes.
\newblock On the tails of waiting time distributions.
\newblock \emph{J. Appl. Prob.}, 7, 745--789, 1975.

\bibitem[Palmowski and Pistorius (2009)]{PalmowskiPistoriusexplev}
Z.~Palmowski and M.~Pistorius.
\newblock Cram\'er asymptotics for finite time first passage probabilities of
  general L\'evy processes.
\newblock \emph{Stat. Prob. Lett.}, 79(16), 1752--1758,
  2009.

\bibitem[Palmowski and Rolski (2002)]{jabernoulli}
Z.~Palmowski and T.~Rolski.
\newblock A technique for the exponential change of measure for Markov
  processes.
\newblock \emph{Bernoulli}, 8(6), 767--785, 2002.

\bibitem[Rolski et~al. (1999)]{rolskibook}
T.~Rolski, H.~Schmidli, V.~Schmidt, and J.~Teugels.
\newblock \emph{Stochastic Processes for Insurance and Finance}.
\newblock Wiley, 1999.

\bibitem[Salah and Morales (2012)]{SalahMorales}
Z.B.~Salah and M.~Morales.
\newblock L\'evy systems and the time value of ruin for Markov additive
  processes.
\newblock \emph{European Actuarial Journal}, 2:\penalty0 289--317, 2012.

\bibitem[Segerdahl (1959)]{Segerdahl}
C.-O.~Segerdahl.
\newblock A survey of results in the collective theory of risk.
\newblock \emph{n Probability and statistics: The Harald Cram\'er volume}, 276--299, 1959.

\bibitem[Zachary (2004)]{ZACH}
S.~Zachary.
\newblock A note on Veraverbeke's theorem.
\newblock \emph{Queueing Systems}, 46, 9--14, 2004.

\bibitem[Zhao and Dong (2018)]{Chineeseparisian}
X.~Zhao and H.~Dong.
\newblock Parisian ruin probability for Markov additive risk processes.
\newblock \emph{Advances in Difference Equations}, 2018, 179, 2018.

\end{thebibliography}

\end{document}